\documentclass{amsart}

\usepackage[utf8]{inputenc}
\usepackage{todonotes}
\usepackage{amsmath}
\usepackage{amsfonts}
\usepackage{amsthm}
\usepackage{amssymb}
\usepackage{enumitem}
\usepackage{caption}
\usepackage{subcaption}
\usepackage[a4paper, top=2cm, left=2.5cm, right=2.5cm, bottom=2.5cm]{geometry}
\usepackage{parskip}
\usepackage{amsaddr}
\usepackage{hyperref}
\usepackage{accents}

\newtheorem{proposition}{Proposition}[section]
\newtheorem{lemma}[proposition]{Lemma}
\newtheorem{theorem}[proposition]{Theorem}

\theoremstyle{remark}
\newtheorem{remark}[proposition]{Remark}

\theoremstyle{definition}
\newtheorem{definition}[proposition]{Definition}

\DeclareMathOperator{\R}{\mathbb{R}}
\DeclareMathOperator{\T}{\mathbb{T}}

\title[BRS vs. MFG] {Comparing the Best Reply strategy and Mean Field Games: The stationary case}
\author{Matt Barker}
\address{Department of Mathematics, Imperial College London, London, SW7 2AZ, UK and Grantham Institute, Imperial College London, London, SW7 2AZ, UK \\
Email: m.barker17@imperial.ac.uk}
\author{Pierre Degond}
\address{Department of Mathematics, Imperial College London, London, SW7 2AZ, UK \\
Email: p.degond@imperial.ac.uk}
\author{Marie-Therese Wolfram}
\address{University of Warwick, Mathematics Institute, Gibbet Hill Road, CV47AL Coventry, UK and Radon Institute for Computational and Applied Mathematics, Altenbergerstr. 69, 4040 Linz, Austria \\
Email: m.wolfram@warwick.ac.uk}

\date{}

\begin{document}

\maketitle

\section*{Acknowledgements}
\noindent MB acknowledges support by the Natural Environment Research Council (NERC) under training grant no. NE/L002515/1. PD acknowledges support by the Engineering and Physical Sciences Research Council (EPSRC) under grants no. EP/M006883/1 and EP/N014529/1, by the Royal Society and the Wolfson Foundation through a Royal Society Wolfson Research Merit Award no. WM130048 and by the National Science Foundation (NSF) under grant no. RNMS11-07444 (KI-Net). PD is on leave from CNRS, Institut de Math\'ematiques de Toulouse, France. MTW  acknowledges  partial  support  from  the  Austrian  Academy  of
Sciences under the New Frontier’s grant NST-001 and the EPSRC under the First Grant EP/P01240X/1. 

\section*{AMS subject classification}
35Q84, 35Q91, 35J15, 35J57, 91A13, 91A23, 49N70, 34C60, 37M05

\section*{Key words}
Mean field games, best reply strategy, stationary Fokker-Planck equation

\section*{Data statement: no new data were collected in the course of this research.}

\section{Abstract}
\noindent Mean field games (MFGs) and the best reply strategy (BRS) are two methods of describing competitive optimisation of systems of interacting agents. The latter can be interpreted as an approximation of the respective MFG system, see ~\cite{Barker,Bertucci2019,Degond2017}. In this paper we present a systematic analysis and comparison of the two approaches in the stationary case. We provide novel existence and uniqueness results for the stationary boundary value problems related to the MFG and BRS formulations, and we present an analytical and numerical comparison of the two paradigms in a variety of modelling situations.

\section{Introduction}

\noindent Mean field games (MFGs) describe the dynamics of large interacting agent systems, in which individuals determine their optimal strategy by minimising a given cost functional. The extensive current literature is based on the original work of Lasry and Lions~\cite{Lasry2006,Lasry2006a,Lasry2007} and Huang, Caines and Malham\'e~\cite{Huang2007a,Huang2006,Huang2006b}. MFGs have been used successfully in many different disciplines. A good overview is presented by Caines, Huang and Malham\'e in~\cite{Caines2017}, a detailed probabilistic approach by Carmona and Delarue in~\cite{Delarue2018,Delarue2018a}.

MFGs can be formulated as parabolic optimal control problems (under certain conditions on the cost). This connection can be used to construct approximations to MFGs. Degond, Liu and Ringhofer proposed a so-called best reply strategy (BRS) in~\cite{Degond2014b}. It can be derived from the corresponding explicit in time discretisation of the respective optimal control problem, as in~\cite{Barker,Degond2017}. More recently it has been derived in~\cite{Bertucci2019} through considering a discounted optimal control problem and taking the discount factor to $\infty$. Specifically, in~\cite{Barker,Degond2017} the limit $\Delta t\rightarrow 0$ in the case of the following cost functional 
\[ J^{\Delta t}(\alpha;m) = \mathbb{E} \left[ \int_t^{t + \Delta t} \left(\frac{\alpha_s^2}{2} + \frac{1}{\Delta t} h(X_s,m(X_s))\right)~ ds \right] \, , \]
is analysed. In~\cite{Bertucci2019} the limiting behavior of MFG systems for cost functionals
\[ J^\rho(\alpha;m) = \mathbb{E} \left[ \int_0^T \left( \frac{\alpha_s^2}{2} + h(X_s,m(X_s)) \right) e^{- \rho s}~ds \right] \, , \]
as the temporal discount factor $\rho$ tends to infinity is considered. In both cases the resulting dynamics depend instantaneously on the cost function $h$. Hence agents do not anticipate future dynamics in the respective limits, as they do in MFG approaches.

As far as the authors are aware, no systematic analysis and comparison of the two approaches has been done yet. However, the BRS is computationally less expensive and therefore more attractive in applications. Therefore it is important to understand under which circumstances the use of each model is appropriate and whether there are situations where the two models are comparable. This paper is a first step analysing the similarities and differences between the two models in a systematic way.

The existence and uniqueness of solutions to stationary MFGs has been studied extensively in previous literature ( c.f.~\cite{Cardaliaguet2015,Ferreira2018,Gomes2017,Lasry2007}). However, apart from a small number of papers e.g.~\cite{Cirant2015,Ferreira2019}, almost all results focus on problems posed on the torus in order avoid dealing with boundary conditions. In~\cite{Benamou2017} the Dirichlet problem was motivated as a stopping time problem, and it was analysed in~\cite{Ferreira2019}. In this paper we consider Neumann boundary conditions, which relate to a no-flux boundary. The only other paper we are aware of that deals with such a situation is~\cite{Cirant2015}. In this paper the authors prove existence of solutions to the MFG problem with non-local dependence on the distribution using a Schauder fixed point argument. They then perturb the solutions to prove existence in the case of a local dependence on the distribution. Other typical methods of proof use continuation methods~\cite{Evangelista2018,Ferreira2018,Gomes2014}, Schauder's fixed point theorem~\cite{Boccardo2016,Cirant2016} or variational approaches through energy minimisation problems~\cite{Cesaroni2018,Evangelista2018a}. In our proof we exploit the linear-quadratic nature of the control. This was done in the time-dependent case in~\cite{Gueant2012} where the problem was reduced to a forward-backward system of heat equations, but we don't think our method has been considered in the stationary case. Our result sits nicely alongside the only other result for Neumann boundary conditions~\cite{Cirant2015}. On the one hand the Hamiltonian used in~\cite{Cirant2015} is more general than ours, however the regularity assumptions and the form of nonlinearity $h$ required in~\cite{Cirant2015} is relaxed in our case.

Due to assumption \ref{a:hincrease}, which states that the running cost $h$ is an increasing function of density, we are in the setting of monotone stationary MFGs. Existence and uniqueness of such MFGs has been studied extensively by Gomes and collaborators in a number of papers e.g~\cite{Evangelista2018a,Gomes2017,Gomes2016}. Although the setting in these papers focusses on domains with periodic boundary conditions, it is worth mentioning the types of techniques used and how they compare to the method in this paper. In \cite{Gomes2017} a Hopf--Cole transformation is used to prove existence and uniqueness of minimisers of an energy functional related to a specific case of an MFG with periodic boundary conditions and a cost $h$ that is logarithmic in the density. The concepts used in our existence and uniqueness proof are similar to those used in \cite{Gomes2017}, though we are able to generalise the density dependence and consider Neumann boundary value problems. In \cite{Gomes2014} the results of \cite{Gomes2017} are extended using a continuation method. There were further improved in \cite{Ferreira2018} where a combination of a continuation method and Minty's method is used. In both cases the methods allow the authors to perturb a problem for which existence and uniqueness is known to prove existence and uniqueness of the problem of interest. The methods used there and in many subsequent works (e.g. \cite{Ferreira2019}) rely on monotonicity properties of the operators. In our work presented here monotonicity similarly plays a central role in proving existence and uniqueness --- through both the use of the maximum principle to prove existence and uniqueness for strictly increasing functions $h$ and through the ability to uniformly perturb an increasing function into a strictly increasing function through the addition of a logarithmic congestion term.

Our non-linear stationary BRS model~\eqref{eq:brssystem} is an example of a stationary non-linear Fokker-Planck equation. Existence and uniqueness of solutions to non-linear Fokker-Planck equations have been studied extensively (see for example~\cite{Carrillo2019,Carrillo2006} and references therein). Many results (e.g. in~\cite{Carrillo2019,Chayes2010,Tugaut2014}) focus on non-local non-linear terms i.e. they consider Fokker-Planck equations of the form
\begin{equation} \label{eq:intro-nonloc-fp}
    - \frac{\sigma^2}{2} \nabla^2 m - \nabla \cdot (m \nabla W*m) = 0 \, ,
\end{equation}
with suitable boundary conditions. Here $\nabla W*m = \int_{\Omega} \nabla W(x - y) m(y)~dy$ is the usual convolution operator. For our model, we consider a local function of density $h = h(x,m(x))$, rather than a convolution term. In this case there are a number of results of existence and uniqueness of solutions to the stationary model, as well as convergence of the dynamic model to the stationary version, see e.g.~\cite{Biane2001,Carrillo2003,Carrillo2006,McCann1997}. These papers all consider the term $h$ to be either independent of $x$ i.e. $h = h(m)$, or of the form $h = h_1(x) + h_2(m)$. So our result extends this case to more general local functions of the density. In previous literature the proof for the local case, as in~\cite{McCann1997}, relies on a related energy functional, for which minimisers can be proven to exist and be unique. Then these minimisers are also solutions to the Fokker-Planck equation. Our result takes a different approach, one that is more closely related to the case with a convolution term, as in~\cite{Carrillo2019}. For the non-local case~\eqref{eq:intro-nonloc-fp} it has been frequently shown (c.f.~\cite{Carrillo2019,Tamura1984}) that solutions of the PDE are equivalent to fixed points of a non-linear map 
\[ T(m) = \frac{1}{Z} e^{- \frac{2}{\sigma^2} W*m} \, , \quad \text{where } Z = \int_{\Omega} e^{- \frac{2}{\sigma^2} W*m}~dx \, . \]
We approach existence and uniqueness of solutions to our PDE~\eqref{eq:brssystem} in a similar vein, considering solutions to the implicit equation~\eqref{eq:xu_brs_sys1}. While the proof in~\cite{Carrillo2019} relies on Schauder’s fixed point theorem, we are able to take advantage of $h$ being a local function of density so instead we use the implicit function theorem and intermediate value theorem to prove our result. 

This paper is organised as follows. We start by briefly introducing the time-dependent MFG and BRS models in Section~\ref{sec:dynamic_setup}, following a more detailed derivation presented in~\cite{Barker,Degond2017}. We then describe how the dynamic problems relate to the stationary case. In Section~\ref{sec:ex_unique_stat_sol} we present a proof of existence and uniqueness for the stationary BRS and MFG. Both proofs use similar arguments for proving existence and uniqueness, relying on the observation that both models involve a stationary Fokker-Planck equation with integral constraints. In Section~\ref{sec:quad potential} we describe an explicit solution to the MFG and BRS model. The explicit solution allows us to analyse in which problem specific parameter ranges the solutions are compareable and in which not.  In Section~\ref{sec:numerical_sim} we illustrate the different behavior of solutions to both models with numerical simulations. Finally, in Section~\ref{sec:conclusion} we conclude with summarising the implications of the results found, specifically what they tell us about the similarities and differences between the two models. We also briefly comment on future directions for research. 

\subsection{The dynamic MFG problem and the corresponding BRS}\label{sec:dynamic_setup}

First we briefly review the underlying modeling assumptions of MFGs and the respective BRS models. For the ease of presentation we consider a quadratic cost on the control and restrict ourselves to the $d$-dimensional torus $\T^d$ in the introduction. However we will consider bounded domains with Neumann boundary conditions from Section~\ref{sec:ex_unique_stat_sol} on. Note that some of the following arguments have not been proven for such a set-up. However it is not unreasonable to assume that the following results extend naturally to the bounded domain case. 
Consider a distribution of agents which is absolutely continuous with respect to the Lebesgue measure. We denote the density of the distribution by a function $m:\T^d \to [0,\infty)$. We take a representative agent, with state $X_t \in \T^d$ moving in this distribution according to the following SDE:
\[ \begin{aligned}
        & dX_t = \alpha_t dt + \sigma dB_t \\
        & \mathcal{L}(X_0) = m_0 \, ,
    \end{aligned} \]
where $\mathcal{L}(X_0)$ denotes the law of the random variable $X_0$, $m_0$ is a given initial distribution of all agents, $\sigma \in (0,\infty)$ denotes the size of idiosyncratic noise in the model and $B_t$ is a $d$-dimensional Wiener process. The function $\alpha_t:[0,T] \to \T^d$ is a control chosen by the representative agent as a result of an optimisation problem, from a set of admissable controls $\alpha \in \mathcal{A}$. The representative agent takes the distribution of other agents, $m$, to be given and attempts to optimise the following functional
\[ J(\alpha;m) = \mathbb{E} \left[ \int_0^T \left(\frac{\alpha_s^2}{2} + h(X_s,m(X_s))\right)~ds \right] \, . \]
The functional $J$ consists of a quadratic cost for the control $\alpha_t$ and a density dependent cost function $h:\T^d \times (0,\infty) \to \R$ over a finite time horizon $T \in (0,\infty)$. Then the optimal cost trajectory $u(x,t)$ is
\[ u(x,t) = \inf_{\alpha \in \mathcal{A}} \mathbb{E} \left[ \left. \int_t^T \left(\frac{\alpha_s^2}{2} + h(X_s,m(X_s))\right)~ds \right| X_t = x \right] \, . \]
The optimal control is given in terms of $u$ by $\alpha_t^* = - \nabla u(X_t,t)$. The optimal cost trajectory evolves backwards in time according to
\begin{gather*}
    \partial_t u = \frac{\left| \nabla u \right|^2}{2} - h(x,m) - \frac{\sigma^2}{2} \nabla^2 u \\
        u(x,T) = 0 \, .
\end{gather*}
We complete the model by assuming all agents act in the same way as the representative agent, and so the backward PDE is coupled to a forward Fokker-Planck PDE describing the evolution of agents. So the full MFG model is given by
\begin{subequations} \label{eq:dynmfg}
    \begin{align}
        & \partial_t u = \frac{\left| \nabla u \right|^2}{2} - h(x,m) - \frac{\sigma^2}{2} \nabla^2 u \\
        & \partial_t m = \nabla \cdot \left[m \nabla u \right] + \frac{\sigma^2}{2} \nabla^2 m\\
        & m(x,0) = m_0 \\
        & u(x,T) = 0 \, . 
    \end{align}
\end{subequations}
Following the approach in~\cite{Barker,Degond2017}, the corresponding BRS model arises through considering a rescaled cost functional over a short, rolling time horizon
\[ J^{\Delta t}(\alpha;m) = \mathbb{E} \left[ \left. \int_t^{t + \Delta t} \left( \frac{\alpha_s^2}{2} + \frac{1}{\Delta t} h(X_s,m(X_s)) \right) ds \right| X_t = x \right] \, . \]
Going through a similar procedure to the MFG problem, approximating the result up to $O(\Delta t)$ and taking the limit $\Delta t \to 0$, the optimal control is given by $\alpha_t = - \left. \left[ \nabla h(x,m(x)) \right] \right|_{x = X_t}$. Again we complete the model by assuming all agents act in the same way as the representative agent. Then the distribution of agents evolves according to the Fokker-Planck equation
\begin{subequations}\label{eq:brs}
    \begin{align}
        \partial_t m &= \nabla \cdot \left[ \left( \nabla h(x,m(x)) \right) m  \right] + \frac{\sigma^2}{2} \nabla^2 m\\
        m(x,0) &= m_0 \, .
    \end{align}
\end{subequations}In \eqref{eq:brs} the dynamics of agents is influenced by the current agent density only. Hence the anticipation behavior, which is characteristic for MFG, is `lost'. Only the current state drives the dynamics. We shall refer to equation~\eqref{eq:brs} as the BRS strategy in the following.  

\subsection{From the dynamic problems to the stationary case} \label{sec:stat_prob}

For the MFG the interpretation of the stationary problem is slightly subtle because in the dynamic case we are considering a problem set on a fixed time horizon so we cannot simply consider the stationary problem by setting $\partial_t u,\partial_t m = 0$ and interpreting it as the long-time behaviour of the dynamic case. Instead we follow the work by Cardaliaguet, Lasry, Lions and Porretta~\cite{Cardaliaguet2012}. To highlight the dependence of the MFG on the time horizon $T$ we use the notation $\bar{u}^T,\bar{m}^T$ for solutions satisfying~\eqref{eq:dynmfg}. Then we define the rescaled functions $u^T$ and $m^T$ by
\[ u^T(x,t) = \bar{u}^T(x,tT) \, , \quad m^T(x,t) = \bar{m}^T(x,tT) \, . \]
Then Theorem 1.2 in~\cite{Cardaliaguet2012} states that under some mild assumptions on the data we have that as $T \to \infty$:
\[ \begin{aligned}
    u^T - \int_{\T^d} u^T~dy \to u \, , \quad & \text{in} \, L^2(\T^d \times (0,1)) \\
    \frac{1}{T} u^T \to (1 - t) \lambda \, , \quad & \text{in} \, L^2(\T^d \times (0,1)) \\
    m^T \to m \, , \quad & \text{in} \, L^p(\T^d \times (0,1)) \, ,
\end{aligned} \]
where $p$ depends on the space dimension $d$. Then the triple $(m,u,\lambda) \in C^2(\T^d) \times C^2(\T^d) \times \R$ satisfies the following stationary problem
\begin{subequations}\label{eq:statmfg}
    \begin{align} 
        - \frac{\sigma^2}{2} \nabla^2 m - \nabla \cdot (m \nabla u) &= 0 \\
        - \frac{\sigma^2}{2} \nabla^2 u + \frac{| \nabla u |^2}{2} - h(x,m) + \lambda &= 0 \\
        \int_{\T^d} m~dx &= 1 \\
        \int_{\T^d} u~dx &= 0 \, .
    \end{align}
\end{subequations}
The corresponding stationary BRS model is obtained by setting $\partial_t m=0$ in~\eqref{eq:brs}. Hence we have
\begin{subequations} \label{eq:statbrs}
    \begin{align}
        \nabla \cdot \left[ \left( \nabla h(x,m(x)) \right) m  \right] + \frac{\sigma^2}{2} \nabla^2 m &= 0 \\
        \int_{\T^d} m~dx &= 1 \, .
    \end{align}
\end{subequations}
Equation~\eqref{eq:statbrs} can be understood as either the long-time behaviour of the dynamic BRS or, under suitable convexity conditions (c.f.~\cite{McCann1997}), a competitive equilibrium distribution of the following minimisation problem
\begin{align}\label{eq:Estat}
\min \mathbb{E} \left[ h(X_t,m(X_t)) \right] 
\end{align}
By competitive equilibrium we mean a stationary distribution $m$ for which $\mathbb{E} \left[ h(X,m(X)) \right]$ is minimised when $\mathcal{L}(X) = m$.

\section{Existence and Uniqueness of Stationary Solutions} \label{sec:ex_unique_stat_sol}
In this section we will show that the MFG~\eqref{eq:statmfg} and the BRS~\eqref{eq:statbrs} admit unique solutions on a bounded domain $\Omega$ with no flux boundary conditions. We make the following assumptions on the function ${h(x,m): \Omega \times (0,\infty) \to \R}$ and the domain $\Omega$:
\begin{enumerate}[label=(A$_\arabic*$)]
\item  \label{a:omega} $\Omega \subset \R^d$ is an open bounded set with a $C^{2,\alpha}$ boundary, for some $\alpha \in (0,1)$ and $d \geq 1$.
    \item $h(x,\cdot)$ is an increasing function for every $x \in \Omega$.\label{a:hincrease}
    \item There exists a continuous function $g:(0,\infty) \to [0,\infty)$ such that $\sup_{x \in \Omega} |h(x,m)| \leq g(m)$ for every $m \in (0,\infty)$. \label{a:hbound}
\end{enumerate}
Since $\Omega$ is bounded, we can now define $|\Omega| = \int_{\Omega} dx$, where this integral is with respect to the standard Lebesgue measure. Furthermore, we denote the unit outer normal vector by $\nu$. \\
For the BRS we further assume:
\begin{enumerate}[label=(BRS$_\arabic*$)]
    \item  $h \in C^2 \left( \Omega \times (0,\infty) \right) \cap C^1 \left( \bar{\Omega} \times(0,\infty) \right)$. \label{a:brs1}
    \item There exists a continuous function $f: (0,\infty) \to [0,\infty)$ such that $\sup_{x \in \Omega} | \nabla_x h(x,m)| \leq f(m)$ for every $m \in (0,\infty)$. \label{a:brs2}
\end{enumerate}
While for the MFG we will assume:
\begin{enumerate}[label=(MFG$_\arabic*$)]
    \item $h \in C \left( \Omega \times (0,\infty) \right)$ \label{a:mfg_hreg}
    \item $\lim_{m \to 0} \sup_{x \in \Omega} h(x,m) < \inf_{x \in \Omega} h \left( x,\frac{1}{|\Omega|} \right)$. \label{a:mfg_mto0}
    \item $\sup_{x \in \Omega} h \left( x,\frac{1}{|\Omega|} \right) < \lim_{m \to \infty} \inf_{x \in \Omega} h(x,m)$. \label{a:mfg_mtoinf}
\end{enumerate}

\subsection*{Discussion of assumptions:} Since we are interested in classical solutions (generally in $C^2\left(\Omega\right) \cap C^1\left(\bar{\Omega}\right)$) the above assumptions on the cost and domain ensure sufficient regularity and boundedness of $h$. It is worth mentioning why we need $h$ to be increasing. This assumption of an increasing function can be related to ``crowd aversion''. When $h$ is increasing in $m$ then areas of high density are more highly penalised than low density areas in the optimisation problem related to the MFG and BRS (see Section~\ref{sec:dynamic_setup}). This prevents ``accumulation points" occurring where higher density is preferable and a Dirac delta might be introduced into the solution. As well as being a problem for regularity, the position of the Dirac deltas would be sensitive on the data and so uniqueness could not be guaranteed.

Assumptions~\ref{a:mfg_mto0} and~\ref{a:mfg_mtoinf} are also less intuitive than the rest of the assumptions. The MFG problem has two integral constraints related to it. We prove that these constraints can be satisfied using the intermediate value theorem. In doing so we show that two functions $m_1,m_2$ exist such that $m_1(x) \leq \frac{1}{|\Omega|} \leq m_2(x)$ for every $x \in \Omega$. The existence of these functions is guaranteed if assumptions~\ref{a:mfg_mto0} and~\ref{a:mfg_mtoinf} hold. As it may not be initially clear what kind of function $h$ satisfies our requirements, some sufficient conditions if $h(x,m) = h_1(x) + h_2(m)$ are:
\begin{itemize}
    \item $h_1 \in C^2\left(\Omega\right) \cap C^1\left(\bar{\Omega}\right)$
    \item $h_2 \in C^2\left((0,\infty)\right)$
    \item $h_2$ is increasing
    \item $\lim_{m \to 0} h_2(m) < h_2\left( \frac{1}{|\Omega|} \right) + \inf_{x \in \Omega} h_1(x)$
    \item $\lim_{m \to \infty} h_2(m) > h_2\left( \frac{1}{|\Omega|} \right) + \sup_{x \in \Omega} h_1(x)$
\end{itemize}

\subsection{Best Reply Strategy} \label{section:xu_brs}
We start by defining the notion of classical solutions we are aiming for. 
\begin{definition}
    Let assumptions~\ref{a:omega}--\ref{a:hbound} and~\ref{a:brs1}--\ref{a:brs2} be satisfied. Then the stationary BRS boundary value problem is to find a function $m: \Omega \to (0,\infty)$ satisfying
    \begin{subequations}\label{eq:brssystem}
        \begin{align}
            m \in C^2\left(\Omega\right) \cap C^1 \left(\bar{\Omega}\right)&\\
            - \frac{\sigma^2}{2} \nabla^2 m - \nabla \cdot (m \nabla [h(x,m)]) &= 0 \, , \quad x \in \Omega \label{eq:xu_brs}\\ 
            - \frac{\sigma^2}{2} \nabla m \cdot \nu - m \nabla [h(x,m)] \cdot \nu &= 0 \, , \quad x \in \partial \Omega \label{eq:xu_brs_bc} \\
            \int_{\Omega} m\, dx &= 1 \, . \label{eq:xu_brs_norm}
        \end{align}
    \end{subequations}
\end{definition}
Throughout this subsection we will assume~\ref{a:omega}--\ref{a:hbound} and~\ref{a:brs1}--\ref{a:brs2} hold.

\begin{lemma} \label{lm:xu_brs_1}
For any $Z \in (0, \infty)$ there exists a unique $m_Z: \Omega \to (0,\infty)$ such that 
\begin{equation} \label{eq:xu_brs_mz}
     m_Z(x) = \frac{1}{Z} e^{- \frac{2}{\sigma^2}h(x,m_Z(x))} \, .
\end{equation}
Furthermore, $m_Z \in C^2 \left( \Omega \right) \cap C^1\left(\bar{\Omega}\right)$.
\end{lemma}
\begin{proof}
    Fix $Z \in (0,\infty)$ and $x \in \Omega$. Consider $G_{Z,x}: (0,\infty) \to \R$ given by
    \[ G_{Z,x}(m) = \frac{1}{Z} e^{-\frac{2}{\sigma^2}h(x,m)} - m \, . \]
    This is a strictly decreasing function of $m$. Furthermore, since $h$ is increasing and continuous with respect to $m$ and $\sup_{x \in \Omega} h(x,m) \leq g(m)$, we must have $\lim_{m \to 0} \sup_{x \in \Omega} h(x,m) \leq \sup_{x \in \Omega} h(x,1) \leq g(1) < \infty$. So we get the following limit inequality as $m \to 0$, which holds uniformly in $x$:
    \[ \lim_{m \to 0} G_{Z,x}(m) > 0 \, . \]
    So there exists some $\epsilon > 0$, independent of $x$, such that $G_{Z,x}(\epsilon) > 0$. Furthermore, after defining a constant ${C = \inf_{x \in \Omega} h(x,\epsilon)}$, it is clear that
    \[ G_{Z,x}\left( \frac{1}{Z} e^{-\frac{2}{\sigma^2}C} + \epsilon \right) \leq \frac{1}{Z} e^{-\frac{2}{\sigma^2} h(x,\epsilon)} - \frac{1}{Z} e^{-\frac{2}{\sigma^2}C} - \epsilon \leq - \epsilon < 0 \, .\]
    Therefore by the intermediate value theorem and strict monotonicity of $G_{Z,x}$ there exists a unique ${m = m_Z(x) > 0}$ such that $G_{Z,x}(m_Z(x)) = 0$. Hence the first result follows. In order to show the regularity requirement that ${m_Z \in C^2\left(\Omega\right) \cap C^1\left(\bar{\Omega}\right)}$, we need to show
    \begin{enumerate}
        \item $m_Z \in C^2\left(\Omega\right)$.
        \item For any $x \in \partial \Omega$, $\lim_{y \to x, \, y \in \Omega} m_Z(y)$ exists.
        \item For any $x \in \partial \Omega$, $\lim_{y \to x, \, y \in \Omega} \nabla m_Z(y)$ exists.
    \end{enumerate}
    The assertion that $m_Z \in C^2(\Omega)$ follows from the implicit function theorem. For the implicit function theorem to hold we require that $G_{Z,x}(m)$ is a $C^2$ function with respect to $x$ and $m$ at $(x,m_Z(x))$ and that $G_{Z,x}'(m_Z(x)) \neq 0$. The first requirement is true from our assumption that $h \in C^2\left(\Omega \times (0,\infty)\right)$, the second requirement is true since
    \[ G_{Z,x}'(m) = - \frac{2}{\sigma^2 Z} \partial_m h(x,m) e^{- \frac{2}{\sigma^2} h(x,m)} - 1 \leq -1 < 0 \, . \]
    To prove that $\lim_{y \to x, \, y \in \Omega} m(y)$ exists for every $x \in \partial \Omega$ it is enough to show $m_Z$ is uniformly Lipschitz in $\Omega$. Since $m_Z \in C^2(\Omega)$ it is therefore enough to show $\| \nabla m_Z\|_{\infty} < \infty$. Note that $C, \epsilon$ defined above are independent of $x$, so we must have $ \epsilon \leq \|m_Z\|_{\infty} \leq \frac{1}{Z} e^{-\frac{2}{\sigma^2}C} + \epsilon < \infty$. Then by differentiating the implicit formula for $m_Z$ we get
    \begin{equation} \label{eq:xu_brs_implicit_grad}
        \nabla m_Z = \frac{- 2 m_Z \nabla_x h(x,m_Z)}{\sigma^2 + 2 m_Z \partial_m h(x,m_Z)} \, .
    \end{equation}
    But $\partial_m h \geq 0$ since $h$ is increasing, also $m_Z$ is uniformly bounded as seen above. Similarly, using $f$ from assumption~\ref{a:brs2} we find $\|\nabla_x h(\cdot,m_Z(\cdot))\|_{\infty} < \infty$, hence $\|\nabla m_Z\|_{\infty} < \infty$.
    
    To prove the final assertion that $\lim_{y \to x, \, y \in \Omega} \nabla m_Z(y)$ exists for any $x \in \partial \Omega$, we note the formula for $\nabla m_Z$ is given by~\eqref{eq:xu_brs_implicit_grad}. Since $m_Z \in C^0\left(\bar{\Omega}\right)$, $\|m_Z\|_{\infty} < \infty$, $\nabla_x h, \partial_m h \in C^0\left(\bar{\Omega} \times (0,\infty)\right)$ and $\sigma^2 + 2 m_Z \partial_m h(x,m_Z) \geq \sigma^2$, then the right hand side of~\eqref{eq:xu_brs_implicit_grad} has limit as $y \to x$ for every $x \in \partial \Omega$. Hence $\nabla m_Z$ does as well.
\end{proof}

\begin{definition}
Let $m_Z$ be given by~\eqref{eq:xu_brs_mz}. Then we define the following function $\Phi:(0, \infty) \to (0,\infty)$
    \[ \Phi(Z) = \int_{\Omega} m_Z~dx \, . \]
\end{definition}

\begin{lemma} \label{lm:xu_brs_2}
    There exists $\bar{Z}, \underaccent{\bar}{Z}$ such that $\Phi\left(\bar{Z}\right) \geq 1$ and $\Phi\left(\underaccent{\bar}{Z}\right) \leq 1$.
\end{lemma}
\begin{proof}
    Take $C_1 = \inf_{x \in \Omega} h \left( x,\frac{1}{|\Omega|} \right)$. So $C_1 \in \left[- g \left( \frac{1}{\Omega} \right),g \left( \frac{1}{\Omega} \right) \right]$. Then, with $\underaccent{\bar}{Z} = |\Omega| e^{- \frac{2}{\sigma^2} C_1} \in (0,\infty)$, we have
    \[ G_{\underaccent{\bar}{Z},x} \left( \frac{1}{|\Omega|} \right) \leq 0 \, \text{ for every} \, x \in \Omega \, . \]
    Hence, $m_{\underaccent{\bar}{Z}}(x) \leq \frac{1}{|\Omega|}$ because $G_{Z,x}$ is a strictly decreasing function. So 
    \[ \Phi\left(\underaccent{\bar}{Z}\right) \leq \|m_{\underaccent{\bar}{Z}}\|_{\infty} |\Omega| \leq 1 \, . \]
    We can similarly find $\bar{Z}$ by taking $C_2 = \sup_{x \in \Omega} h \left( x,\frac{1}{|\Omega|} \right)$. So $C_2 \in \left[- g \left( \frac{1}{\Omega} \right),g \left( \frac{1}{\Omega} \right) \right]$. Then, with $\bar{Z} = |\Omega| e^{- \frac{2}{\sigma^2} C_2} \in (0,\infty)$, we have
    \[ G_{\bar{Z},x} \left( \frac{1}{|\Omega|} \right) \geq 0 \, \text{ for every} \, x \in \Omega \, . \]
    Hence, $m_{\bar{Z}}(x) \geq \frac{1}{|\Omega|}$ because $G_{Z,x}$ is a strictly decreasing function. So 
    \[ \Phi\left(\bar{Z}\right) \geq \|m_{\bar{Z}}\|_{\infty} |\Omega| \geq 1 \, . \]
\end{proof}

\begin{lemma} \label{lm:xu_brs_3}
    There exists a unique $Z^* \in (0,\infty)$ such that $\Phi\left(Z^*\right) = 1$.
\end{lemma}

\begin{proof}
    If $\Phi$ is continuous and strictly decreasing the intermediate value theorem and the Lemma~\ref{lm:xu_brs_2} give the result. We start by proving that $\Phi$ is strictly decreasing. First note that if $m_Z(x)$ is strictly decreasing in $Z$ for every $x$ then $\Phi$ must be strictly decreasing because $m_Z$ is continuous with respect to $x$. Take $Z_1 < Z_2$. Then $m_{Z_1}(x)$ satisfies
    \[ \frac{1}{Z_1} e^{-\frac{2}{\sigma^2}h(x,m_{Z_1}(x))} - m_{Z_1}(x) = 0 \, . \]
    So
    \[ \frac{1}{Z_2} e^{-\frac{2}{\sigma^2}h(x,m_{Z_1}(x))} - m_{Z_1}(x) < \frac{1}{Z_1} e^{-\frac{2}{\sigma^2}h(x,m_{Z_1}(x))} - m_{Z_1}(x) = 0 \, . \]
Hence $G_{Z_2,x}(m_{Z_1}(x)) < 0$. Then $m_{Z_2}(x) < m_{Z_1}(x)$ for all $x \in \Omega$ since $G_{Z,x}$ is a strictly decreasing function. 
To show that $\Phi$ is continuous at $Z \in (0,\infty)$, take $\epsilon < Z' < Z$. Then
    \[ \begin{aligned}
        |\Phi(Z) - \Phi(Z')| & = \Phi(Z') - \Phi(Z) = \int_{\Omega} m_{Z'}(x)~dx - \Phi(Z) \\
        & = \frac{Z}{Z'} \int_{\Omega} \frac{1}{Z} e^{- h(x,m_{Z'}(x))}~dx - \Phi(Z) \\
        & \leq \frac{Z}{Z'} \int_{\Omega} \frac{1}{Z} e^{- h(x,m_Z(x))}~dx - \Phi(Z) \leq \frac{\Phi(Z)}{\epsilon} (Z - Z') \leq \frac{\Phi(\epsilon)}{\epsilon} (Z - Z') \, .
    \end{aligned} \]
By exchanging $Z$ and $Z'$ we can similarly show the analogous result for $\epsilon < Z < Z'$, therefore $\Phi$ is locally Lipschitz and hence continuous.
\end{proof}
\begin{theorem}
    There exists a unique solution $m: \Omega \to (0,\infty)$ to the stationary BRS~\eqref{eq:brssystem}.
\end{theorem}
\begin{proof}
    Take $m(x) = m_{Z^*}(x)$, with $Z^*$ defined as in Lemma~\ref{lm:xu_brs_3}. Then from Lemmas~\ref{lm:xu_brs_1} and~\ref{lm:xu_brs_3} we have shown there exists a unique $m: \Omega \to (0,\infty)$ satisfying
    \begin{subequations} \label{eq:xu_brs_sys1}
        \begin{align}
            & m \in C^2\left(\Omega\right) \cap C^1\left(\bar{\Omega}\right) \\
            & \text{There exists } Z \in (0,\infty) \text{ such that } m = \frac{1}{Z} e^{- \frac{2}{\sigma^2}h(x,m)} \, , \quad x \in \Omega \\
            & \int_{\Omega} m~dx = 1 \, .
        \end{align}
    \end{subequations}
    Now for any $m:\Omega \to (0,\infty)$ we can define $\phi(m)$ by $\phi(m) = e^{- \frac{2}{\sigma^2} h(x,m)}$. Suppose $m$ is a solution to~\eqref{eq:brssystem}, then $\frac{m}{\phi(m)} = m e^{\frac{2}{\sigma^2} h(x,m)}$ and so $\frac{m}{\phi(m)} \in H^1\left(\Omega\right)$ because $m \in C^1\left(\bar{\Omega}\right)$, $h \in C^1\left(\bar{\Omega} \times (0,\infty)\right)$ and $h$ is increasing in $m$. Therefore a solution to~\eqref{eq:brssystem} is equivalent to a solution of
    \begin{subequations} \label{eq:xu_brs_sys2}
        \begin{align}
            m \in C^2\left(\Omega\right) \cap C^1\left(\bar{\Omega}\right)& \label{eq:xu_brs_sys2_reg1} \\
            \frac{m}{\phi(m)} \in H^1(\Omega)& \label{eq:xu_brs_sys2_reg2}\\
            \nabla \cdot \left(\phi(m) \nabla \left( \frac{m}{\phi(m)} \right)\right) &= 0 \, , \quad x \in \Omega \label{eq:xu_brs_sys2_pde}\\
            \phi(m) \nabla \left( \frac{m}{\phi(m)} \right) \cdot \nu &= 0 \, , \quad x \in \partial \Omega \label{eq:xu_brs_sys2_bc}\\
            \int_{\Omega} m~dx &= 1 \, .
        \end{align}
    \end{subequations}
    Now if we multiply~\eqref{eq:xu_brs_sys2_pde} by $\frac{m}{\phi(m)}$, and integrate over $\Omega$, then using Green's formula and the boundary condition~\eqref{eq:xu_brs_sys2_bc} we get
    \[ 0 = \int_{\Omega} \frac{m}{\phi(m)} \nabla \cdot \left(\phi(m) \nabla \left(\frac{m}{\phi(m)} \right)\right)~dx = - \int_{\Omega} \phi(m) \left| \nabla \left(\frac{m}{\phi(m)} \right) \right|^2~dx \, . \]
    But $\phi(m) > 0$ and $\left| \nabla \left(\frac{m}{\phi(m)} \right) \right|^2 \geq 0$ for every $x \in \Omega$. Hence this is only true if $\nabla \left(\frac{m}{\phi(m)} \right) = 0$ for every $x \in \Omega$, i.e. if there exists $Z \in (0,\infty)$ such that $m = \frac{1}{Z}\phi(m)$. Conversely, if $m \in C^2\left(\Omega\right) \cap C^1\left(\bar{\Omega}\right)$ and there exists $Z \in (0,\infty)$ such that $m = \frac{1}{Z}\phi(m)$, then $m$ satisfies~\eqref{eq:xu_brs_sys2_reg1}--\eqref{eq:xu_brs_sys2_bc}. So a solution of~\eqref{eq:xu_brs_sys2} is equivalent to a solution of~\eqref{eq:xu_brs_sys1}. Therefore the systems~\eqref{eq:brssystem} and~\eqref{eq:xu_brs_sys1} are equivalent. Hence we have shown existence and uniqueness of solutions to~\eqref{eq:xu_brs_sys1} by proving existence and uniqueness of solutions to~\eqref{eq:brssystem}.
\end{proof}

\subsection{Mean Field Games}
Next we discuss existence and uniqueness of classical solutions to~\eqref{eq:statmfg}, which is defined as follows.
\begin{definition}
    The stationary MFG boundary value problem is to find $m:\Omega \to (0,\infty)$, $u:\Omega \to \R$ and $\lambda \in \R$ satisfying the following PDE system
    \begin{subequations} \label{eq:xu_mfg}
    \begin{align}
        m \in C^2\left(\Omega\right) \cap C^1\left(\bar{\Omega}\right)& \label{eq:xu_mfg_mreg} \\
        u \in C^2\left(\Omega\right) \cap C^1\left(\bar{\Omega}\right)& \label{eq:xu_mfg_ureg}\\
        - \frac{\sigma^2}{2} \nabla^2 m - \nabla \cdot (m \nabla u) &= 0 \, , \quad x \in \Omega \label{eq:xu_mfg_pde1} \\
        - \frac{\sigma^2}{2} \nabla^2 u + \frac{| \nabla u |^2}{2} - h(x,m) + \lambda &= 0 \, , \quad x \in \Omega \label{eq:xu_mfg_pde2}\\
            - \frac{\sigma^2}{2} \nabla m \cdot \nu &= 0 \, , \quad x \in \partial \Omega \label{eq:xu_mfg_bc1} \\
            - \nabla u \cdot \nu &= 0 \, , \quad x \in \partial \Omega \label{eq:xu_mfg_bc2} \\
            \int_{\Omega} m~dx &= 1, \label{eq:xu_mfg_ic1}\\
            \int_{\Omega} u~dx &= 0 \, . \label{eq:xu_mfg_ic2}
    \end{align}
\end{subequations}
\end{definition}

\begin{remark} \label{remark:xu_mfg_msol}
    Here, following the method of Section~\ref{section:xu_brs}, we note that for any $u \in C^2\left(\Omega\right) \cap C^1\left(\bar{\Omega}\right)$, a solution $m$ of~\eqref{eq:xu_mfg_mreg},~\eqref{eq:xu_mfg_pde1},~\eqref{eq:xu_mfg_bc1},~\eqref{eq:xu_mfg_ic1} is equivalent to a solution of
    \begin{subequations} \label{eq:xu_mfg_mvar}
        \begin{align}
            &m \in C^2\left(\Omega\right) \cap C^1\left(\bar{\Omega}\right)\\
            m &= \frac{1}{Z} e^{- \frac{2}{\sigma^2} u} \, , \quad x \in \Omega \\
            Z &= \int_{\Omega} e^{- \frac{2}{\sigma^2} u}~dx \, . \label{eq:xu_mfg_mvar_ic}
        \end{align}
    \end{subequations}
    Then by the arguments in Section~\ref{section:xu_brs}, a unique $m$ satisfying~\eqref{eq:xu_mfg_mvar} exists and is the unique solution to~\eqref{eq:xu_mfg_mreg},~\eqref{eq:xu_mfg_pde1},~\eqref{eq:xu_mfg_bc1},~\eqref{eq:xu_mfg_ic1}. So from now we only consider that solution.
\end{remark}

\begin{proposition} \label{proposition:xu_mfg_transform}
    There exists a unique solution $(m,u,\lambda) \in \left[ C^2\left(\Omega\right) \cap C^1\left(\bar{\Omega}\right) \right] \times \left[ C^2\left(\Omega\right) \cap C^1\left(\bar{\Omega}\right) \right] \times \R$ to the stationary MFG boundary value problem if and only if there exists a unique solution $(u,\lambda,Z) \in \left[ C^2\left(\Omega\right) \cap C^1\left(\bar{\Omega}\right) \right] \times \R \times (0,\infty)$ to
    \begin{subequations}\label{eq:xu_mfg_sys}
    \begin{align}
        u \in C^2\left(\Omega\right) \cap C^1\left(\bar{\Omega}\right) \label{eq:xu_mfg_bvp1}&\\
        - \frac{\sigma^2}{2} \nabla^2 u + \frac{| \nabla u |^2}{2} - h\left( x,\frac{1}{Z} e^{- \frac{2}{\sigma^2} u} \right) + \lambda &= 0 \, , \quad x \in \Omega \label{eq:xu_mfg_bvp2} \\
        - \nabla u \cdot \nu &= 0 \, , \quad x \in \partial \Omega \label{eq:xu_mfg_bvp3} \\ 
        \int_{\Omega} \frac{1}{Z} e^{- \frac{2}{\sigma^2} u}~dx &= 1, \label{eq:xu_mfg_bvp4} \\
        \int_{\Omega} u~dx &= 0 \label{eq:xu_mfg_bvp5} \, .
    \end{align}
    \end{subequations}
\end{proposition}

\begin{proof}
    First assume a unique solution $(m,u,\lambda)$ to~\eqref{eq:xu_mfg} exists, then thanks to remark~\ref{remark:xu_mfg_msol}, we have $m = \frac{1}{Z} e^{-\frac{2}{\sigma^2}u}$, for $Z$ satisfying~\eqref{eq:xu_mfg_mvar_ic}. Then the triple $(u,\lambda,Z)$ is clearly a solution to~\eqref{eq:xu_mfg_sys}. Furthermore, suppose another solution $(u',\lambda',Z')$ to~\eqref{eq:xu_mfg_sys} exists. Then $(m',u',\lambda')$, with $m' = \frac{1}{Z} e^{- \frac{2}{\sigma^2} u'}$, is a solution to~\eqref{eq:xu_mfg}. But since we assumed such solutions are unique, we have $(m',u',\lambda') = (m,u,\lambda)$ and hence $(u',\lambda',Z') = (u,\lambda,Z)$, so the solution to~\eqref{eq:xu_mfg_sys} is unique.\\
    Next we assume that a unique solution $(u,\lambda,Z)$ to~\eqref{eq:xu_mfg_sys} exists. Then, defining $m = \frac{1}{Z} e^{- \frac{2}{\sigma^2} u}$, $(m,u,\lambda)$ is a solution to~\eqref{eq:xu_mfg}. Now suppose $(m',u',\lambda')$ is another solution then (again using remark~\ref{remark:xu_mfg_msol}) $m' = \frac{1}{Z'} e^{- \frac{2}{\sigma^2} u'}$, where $Z'$ satisfies~\eqref{eq:xu_mfg_mvar_ic}. So $(u',\lambda',Z')$ satisfies~\eqref{eq:xu_mfg_sys}. By uniqueness $(u',\lambda',Z') = (u,\lambda,Z)$ and so $(m,u,\lambda)$ is also unique.
\end{proof}

\begin{theorem} \label{thm:xu_mfg1}
    There exists a unique solution $(m,u,\lambda)$ of the MFG system~\eqref{eq:xu_mfg}.
\end{theorem}

\begin{proof}[Proof (outline)]
    First note, as a result of Proposition~\ref{proposition:xu_mfg_transform}, we only need to prove existence and uniqueness of a solution to~\eqref{eq:xu_mfg_sys} and existence and uniqueness for the MFG system~\eqref{eq:xu_mfg} will follow. The proof is split into the following steps:
    \begin{enumerate}
        \item Show that for any pair of constants $(\lambda,Z)$ there exists a unique solution, denoted by $u_{\lambda,Z}$, to~\eqref{eq:xu_mfg_bvp1},~\eqref{eq:xu_mfg_bvp2}  (see Proposition \ref{prop:xu_mfg_cont})
        \item Show that for any constant $Z$ there exists a unique $\lambda = \lambda(Z)$ such that $u_{\lambda(Z),Z}$ satisfies~\eqref{eq:xu_mfg_bvp4} (see Proposition \ref{prop:xu_mfg_lambda})
        \item Show that there exists a unique $Z = Z^*$ such that $u_{\lambda(Z^*),Z^*}$ satisfies~\eqref{eq:xu_mfg_bvp5}.
    \end{enumerate}
    
    Then $\left(u_{\lambda\left(Z^*\right),Z^*},\lambda\left(Z^*\right),Z^*\right)$ is a solution to~\eqref{eq:xu_mfg_sys}. Uniqueness follows from uniqueness obtained at each step of the proof outlined. We prove step 1 using a variant of the method of upper and lower solutions in the spirit of~\cite{Schmitt1978}, so that it applies to our case of Neumann boundary conditions. We prove steps 2 and 3 by iteratively using the intermediate value theorem - in a similar manner to the proof of Lemma~\ref{lm:xu_brs_3}. Note that we will first do this for $h$ which is strictly increasing in $m$. Then by considering ${h_{\epsilon}(x,m) = h(x,m) + \epsilon \log\left(|\Omega| m\right)}$, and taking the limit as $\epsilon \to 0$ we will prove it in the more general setting when $h$ is increasing.
\end{proof}

\begin{lemma} \label{lm:xu_mfg_constbound}
    There exists $\Lambda_1, \Lambda_2 \in [- \infty, \infty]$ with $\Lambda_1 < \Lambda_2$ such that for every $\lambda \in \left( \Lambda_1, \Lambda_2 \right)$ and $Z > 0$, there exist two constants $\underaccent{\bar}{u}_{\lambda,Z} \leq 0 \leq \bar{u}_{\lambda,Z}$ satisfying 
    \begin{equation} \label{eq:xu_mfg_constbound}
        - h \left(x, \frac{1}{Z} e^{- \frac{2}{\sigma^2} \underaccent{\bar}{u}_{\lambda,Z}} \right) + \lambda \leq 0 \leq - h \left(x, \frac{1}{Z} e^{- \frac{2}{\sigma^2} \bar{u}_{\lambda,Z}} \right) + \lambda \, .
    \end{equation}
\end{lemma}

\begin{proof}
     Take $\Lambda_1 = \lim_{m \to 0} \sup_{x \in \Omega} h(x,m)$ and $\Lambda_2 = \lim_{m \to \infty} \inf_{x \in \Omega} h(x,m)$. First $\Lambda_1 < \Lambda_2$ by combining assumptions~\ref{a:mfg_mto0} and~\ref{a:mfg_mtoinf}. Then, since $h$ is continuous and increasing in $m$, for any $\lambda \in \left( \Lambda_1, \Lambda_2 \right)$ there exists $M_{\lambda}^1, M_{\lambda}^2 \in (0,\infty)$ such that $h(x,m) \leq \lambda$ for all $(x,m) \in \Omega \times \left(0,M_{\lambda}^1\right]$, and similarly $h(x,m) \geq \lambda$ for all $(x,m) \in \Omega \times \left[M_{\lambda}^2,\infty\right)$. We define the upper and lower constants for $\lambda \in \left( \Lambda_1, \Lambda_2 \right)$ as 
     \[ \begin{aligned}
            \bar{u}_{\lambda,Z} &= \max \left( - \frac{\sigma^2}{2} \log Z M_{\lambda}^1, 0 \right) \\
            \underaccent{\bar}{u}_{\lambda,Z} &= \min \left( - \frac{\sigma^2}{2} \log Z M_{\lambda}^2, 0 \right) \, .
        \end{aligned} \]
    Then clearly 
    \[ - h \left( x, \frac{1}{Z} e^{- \frac{2}{\sigma^2}\bar{u}_{\lambda,Z}} \right) + \lambda = - h \left( x,\min \left(M_{\lambda}^1, 1 \right) \right) + \lambda \geq - h \left( x,M_{\lambda}^1 \right) + \lambda \geq 0 \, , \]
    while the reverse inequality is true for $\underaccent{\bar}{u}_{\lambda,Z}$. Hence $\bar{u}_{\lambda,Z}, \underaccent{\bar}{u}_{\lambda,Z}$ are the required upper and lower constants.
\end{proof}

\begin{proposition} \label{prop:xu_mfg_1}
    Define $C^{2,\tau}\left(\bar{\Omega}\right)$ as the set of functions $u \in C^2\left(\bar{\Omega}\right)$ whose second partial derivatives are all H\"older continuous with exponent $\tau$ on $\bar{\Omega}$. Assume $h$ is strictly increasing with respect to $m$. Then, for every $\lambda \in (\Lambda_1, \Lambda_2)$ and every $Z \in (0,\infty)$ there exists a unique function, $u_{\lambda,Z} \in C^{2,\tau}\left(\bar{\Omega}\right) \subset C^2\left(\Omega\right) \cap C^1\left(\bar{\Omega}\right)$ for some $\tau \in (0,1)$, which satisfies~\eqref{eq:xu_mfg_bvp1}--\eqref{eq:xu_mfg_bvp3}. Furthermore, $\underaccent{\bar}{u}_{\lambda,Z} \leq u_{\lambda,Z} \leq \bar{u}_{\lambda,Z}$.
\end{proposition}

\begin{proof}
    Existence is an application of Corollary 2.9 in~\cite{Schmitt1978}, which states that a solution ${u_{\lambda,Z} \in C^{2,\tau}\left(\bar{\Omega}\right)}$ to~\eqref{eq:xu_mfg_bvp1}--\eqref{eq:xu_mfg_bvp3} exists provided the following properties hold:
    
    \begin{enumerate}
        \item There exist constants $\underaccent{\bar}{u}_{\lambda,Z} \leq 0 \leq \bar{u}_{\lambda,Z}$ satisfying~\eqref{eq:xu_mfg_constbound} for every $x \in \bar{\Omega}$.
        \item There exists a continuous function $f:[0,\infty) \to [0,\infty)$ such that the following inequality holds for every $(x,u,p) \in \bar{\Omega} \times \R \times \R^d$
        \[ \left| \frac{|p|^2}{2} - h \left( x,\frac{1}{Z} e^{- \frac{2}{\sigma^2}u} \right) + \lambda \right| \leq f(|u|) \left( 1 + |p|^2 \right) \, . \]
    \end{enumerate}
    
    Property 1 is true from Lemma~\ref{lm:xu_mfg_constbound}. Property 2 can be shown to be true by taking
    \[ f(u) = \max \left( \frac{1}{2}, |\lambda| + \max \left[ g \left( \frac{1}{Z} e^{- \frac{2}{\sigma^2}u} \right), g \left( \frac{1}{Z} e^{\frac{2}{\sigma^2}u} \right) \right] \right) \, , \]
    where $g$ is defined in assumption~\ref{a:hbound}.
    
    We prove uniqueness using the strong maximum principle and Hopf's Lemma as stated in~\cite{Evans1998}~(Section~6.4.2.~pp.~330--333). Suppose there are two solutions, $u_1,u_2 \in C^2\left(\Omega\right) \cap C^1\left(\bar{\Omega}\right)$ to~\eqref{eq:xu_mfg_bvp1},~\eqref{eq:xu_mfg_bvp2}. Define $a = \nabla (u_1 + u_2)$. Then $a \in L^{\infty}\left(\bar{\Omega}\right)$. Now suppose $u_1 \neq u_2$ and define $v = u_1 - u_2$. Then $v$ must attain its maximum at some point $\bar{x} \in \bar{\Omega}$. First suppose $\bar{x} \in \Omega$. Then there exists an open bounded and connected region $V$ such that $V \subset \Omega$, $\bar{x} \in V$ and $v > 0$ for all $x \in V$. Hence, since $h(x,\cdot)$ is increasing, we have
    \[ - \frac{\sigma^2}{2} \nabla^2 v + \frac{1}{2} a \cdot \nabla v \leq 0 \, , \quad \text{for every} \, x \in V \, . \]
    So by the strong maximum principle $v$ is constant in $V$. Therefore we must have
    \[ h \left( x,\frac{1}{Z} e^{- \frac{2}{\sigma^2}u_1(x)} \right) = h \left( x,\frac{1}{Z} e^{- \frac{2}{\sigma^2}u_2(x)} \right)  \quad \text{for every} \, x \in V \, . \]
    So $u_1 = u_2$ in $V$ because $h$ is strictly increasing, which leads to a contradiction. Therefore the only other option is that $\bar{x} \in \partial \Omega$ and $v(x) < v(\bar{x})$ for every $x \in \Omega$. Hence by Hopf's Lemma (which we can use because $\partial \Omega$ is $C^2$) $\left. \frac{\partial v}{\partial \nu} \right|_{\bar{x}} > 0$, but by the boundary condition~\eqref{eq:xu_mfg_bvp3}, $\frac{\partial v}{\partial \nu} = \frac{\partial u_1}{\partial \nu} - \frac{\partial u_2}{\partial \nu} = 0$. This again leads to a contradiction. Therefore $u_1 = u_2$, and therefore the solution is unique.
\end{proof}

\begin{remark} \label{rmk:XU_mfg_r2}
    It should be noted that the same method to prove uniqueness can be used to prove that $u_{\lambda_1,Z} \geq u_{\lambda_2,Z}$ for all $\lambda_1 \leq \lambda_2$
\end{remark}

\begin{lemma} \label{lm:xu_mfg_monotone}
    Assume $h$ is strictly increasing with respect to $m$. Then for every $x \in \Omega$, $u_{\lambda,Z}(x)$ is decreasing with respect to $\lambda$ and $Z$.
\end{lemma}

\begin{proof}
    In Light of remark~\ref{rmk:XU_mfg_r2} we need only to prove $u_{\lambda,Z_1} \geq u_{\lambda,Z_2}$ for all $Z_1 \leq Z_2$. However, by substitution we find that $u = u_{\lambda,Z_1} - \frac{\sigma^2}{2} \log \frac{Z_2}{Z_1}$ satisfies~\eqref{eq:xu_mfg_bvp1}--\eqref{eq:xu_mfg_bvp3} with $Z = Z_2$. So by uniqueness of solutions to this PDE proved in Proposition~\ref{prop:xu_mfg_1} we see that
    \[ u_{\lambda,Z_2} = u \leq u_{\lambda,Z_1} \, . \]
\end{proof}

\begin{proposition} \label{prop:xu_mfg_cont}
    Define $\Phi: (\Lambda_1, \Lambda_2) \times (0,\infty) \to L^{\infty}(\Omega)$ by
    \[ \Phi(\lambda,Z) = u_{\lambda,Z} \, , \]
    where $u_{\lambda,Z}$ is the unique solution to~\eqref{eq:xu_mfg_bvp1}--\eqref{eq:xu_mfg_bvp3} as found in the Proposition~\ref{prop:xu_mfg_1}. Assume $h$ is strictly increasing with respect to $m$. Then $\Phi$ is continuous (with respect to $L^{\infty}$ norm).
\end{proposition}
\begin{proof}

    We will prove $\Phi$ is sequentially continuous. Let $(\lambda_n,Z_n)$ be a sequence in $(\Lambda_1,\Lambda_2) \times (0,\infty)$ that converges to $(\lambda,Z) \in (\Lambda_1,\Lambda_2) \times (0,\infty)$. We consider two sequences: $(\lambda_{n}^{(i)}, Z_{n}^{(i)})$ for $i = 1,2$, which we use to sandwich our original sequence. We set these sequences with the following conditions 
    \begin{enumerate}
        \item $\lambda_n^{(1)} = \inf_{j \geq n} \lambda_j$
        \item $\lambda_n^{(2)} = \sup_{j \geq n} \lambda_j$
        \item $Z_n^{(1)} = \inf_{j \geq n} Z_j$
        \item $Z_n^{(2)} = \sup_{j \geq n} Z_j$
    \end{enumerate}
    In the first part of this proof we show that for each $i = 1,2$, there exists a subsequence $(\lambda_{n_k}^{(i)}, Z_{n_k}^{(i)})$ such that $u_{\lambda_{n_k}^{(i)}, Z_{n_k}^{(i)}} \to u_{\lambda, Z}$. We will only show this for $i = 1$ as the case $i = 2$ is identical. Clearly the sequence $(\lambda_{n}^{(1)}, Z_{n}^{(1)})$ also converges to $(\lambda,Z)$. So there exists a subsequence $n_k$ such that $u_{\lambda_{n_k}^{(1)},Z_{n_k}^{(1)}} \to u_*$ in $C^2(\Omega) \cap C^1(\bar{\Omega})$ because $u_{\lambda_{n_k}^{(1)},Z_{n_k}^{(1)}} \in C^{2,\tau}(\bar{\Omega})$ (by Proposition~\ref{prop:xu_mfg_1}), which is compactly embedded in $C^2(\bar{\Omega}) \subset C^2(\Omega) \cap C^1(\bar{\Omega})$. Therefore we also get the following pointwise convergence
    \begin{align*}
            0 & = \lim_{k \to \infty} \left[-\frac{\sigma^2}{2} \nabla^2 u_{\lambda_{n_k},Z_{n_k}} + \frac{1}{2}\left| \nabla u_{\lambda_{n_k},Z_{n_k}}\right|^2 - h \left( x, \frac{1}{Z_{n_k}} e^{- \frac{2}{\sigma^2} u_{\lambda_{n_k},Z_{n_k}}} \right) + \lambda_{n_k} \right]\\
            & = - \frac{\sigma^2}{2} \nabla^2 u_* + \frac{\left| \nabla u_* \right|^2}{2} - h \left( x, \frac{1}{Z} e^{- \frac{2}{\sigma^2} u_*} \right) + \lambda\\
         0 &= \lim_{k \to \infty} \nabla u_{\lambda_{n_k},Z_{n_k}} \cdot \nu |_{x \in \partial \Omega} = \nabla u_* \cdot \nu |_{x \in \partial \Omega} \, .
    \end{align*}
    So $u_* = u_{\lambda,Z}$, by uniqueness proved in Proposition~\ref{prop:xu_mfg_1}. Now by design we have $\lambda_{n_k}^{(2)} \geq \lambda_n \geq \lambda_{n_k}^{(1)}$ for all $n \geq n_k$ and similarly for $Z_n$, hence $u_{\lambda_{n_k}^{(1)},Z_{n_k}^{(1)}} \geq u_{\lambda_n,Z_n} \geq u_{\lambda_{n_k}^{(2)},Z_{n_k}^{(2)}}$ by Lemma~\ref{lm:xu_mfg_monotone}. So $u_{\lambda_n,Z_n} \to u_{\lambda,Z}$ in $L^{\infty}(\Omega)$.
\end{proof}

\begin{proposition} \label{prop:xu_mfg_lambda}
    For each $Z \in (0,\infty)$ define $I_1(\cdot;Z): (\Lambda_1,\Lambda_2) \to \R$ by
    \[ I_1(\lambda;Z) = \int_{\Omega} u_{\lambda,Z}~dx = \int_{\Omega} \Phi(\lambda,Z)~dx \, . \]
    Assume $h$ is strictly increasing with respect to $m$. Then for every $Z \in (0,\infty)$ there exists a unique $\lambda = \lambda(Z)$ such that $I_1(\lambda(Z);Z) = 0$, furthermore 
    \begin{equation} \label{eq:xu_mfg_lambdabound}
        \inf_{x \in \Omega} h \left( x,\frac{1}{Z} \right) \leq \lambda(Z) \leq \sup_{x \in \Omega} h \left( x,\frac{1}{Z} \right) \, .
    \end{equation}
\end{proposition}

\begin{proof}
    We use the intermediate value theorem to prove this proposition. There are three parts we have to prove
    \begin{enumerate}
        \item For every $Z \in (0,\infty)$ there exists $\lambda_1 \leq \lambda_2 \in (\Lambda_1, \Lambda_2)$ such that $I_1(\lambda_1;Z) \leq 0$ and $I_1(\lambda_2;Z) \geq 0$.
        \item $I_1(\lambda;Z)$ is continuous with respect to $\lambda$ in $[\lambda_1,\lambda_2]$.
        \item $I_1(\lambda;Z)$ is strictly decreasing with respect to $\lambda$.
    \end{enumerate}
    
    Part (1) and part (2) allow us to use the intermediate value theorem to show that for every $Z \in (0,\infty)$ there exists some $\lambda$ such that $I_1(\lambda;Z) = 0$. Part (3) shows that this $\lambda$ is unique, so the function $Z \mapsto \lambda(Z)$ is well defined.
    
    Part (1): Take $\lambda_1 = \sup_{x \in \Omega} h \left( x,\frac{1}{Z} \right) > \Lambda_1$. Then recall that $\bar{u}_{\lambda_1, Z} = \max \left( - \frac{\sigma^2}{2} \log(Z M_{\lambda_1}^1), 0 \right)$, where $M_{\lambda_1}^1$ satisfies $h \left( x,M_{\lambda_1}^1 \right) \leq \lambda_1$. But we can take $M_{\lambda_1}^1 = \frac{1}{Z}$ by our choice of $\lambda_1$. So $u_{\lambda_1,Z} \leq \bar{u}_{\lambda_1, Z} = 0$, and therefore $I_1(\lambda_1;Z) \leq 0$. The choice for $\lambda_2$ is $\lambda_2 = \inf_{x \in \Omega} h \left( x,\frac{1}{Z} \right)$ and the proof is similar to the above. 
    
    Part (2): Take $\lambda_1,\lambda_2$ as above. By Propositions~\ref{prop:xu_mfg_1} and~\ref{prop:xu_mfg_cont} and Lemma~\ref{lm:xu_mfg_monotone}, $u_{\lambda,Z}$ is continuous with respect to $\lambda$ in $L^{\infty}(\Omega)$ and $\underaccent{\bar}{u}_{\lambda_2,Z} \leq u_{\lambda,Z} \leq \bar{u}_{\lambda_1,Z}$ for any $\lambda \in [\lambda_1,\lambda_2]$. So by the dominated convergence theorem $I_1$ is continuous in $\lambda$.
    
    Part (3): Take $\lambda_1 < \lambda_2$, from Lemma~\ref{lm:xu_mfg_monotone} we know $u_{\lambda_1,Z} \geq u_{\lambda_2,Z}$. Clearly, since solutions to the PDE~\eqref{eq:xu_mfg_bvp1}--\eqref{eq:xu_mfg_bvp3} are unique, there exists $a \in \Omega$ such that $u_{\lambda_1,Z}(a) \neq u_{\lambda_2,Z}(a)$. Hence, $u_{\lambda_1,Z}(a) > u_{\lambda_2,Z}(a)$ and so by continuity $I_1(\lambda_1,Z) > I_1(\lambda_2,Z)$.
\end{proof}
    
\begin{remark}
    This proposition ensures that for any $Z \in (0,\infty)$ we can find $\lambda = \lambda(Z)$ and $u = u_{\lambda(Z),Z}$ satisfying~\eqref{eq:xu_mfg_bvp1}--\eqref{eq:xu_mfg_bvp3},~\eqref{eq:xu_mfg_bvp5}, so we are left to find $Z^*$ such that~\eqref{eq:xu_mfg_bvp4} holds.
\end{remark}

\begin{lemma} \label{lm:xu_lambda_monotone}
    Assume $h$ is strictly increasing with respect to $m$. Then the function $\lambda(Z)$ is strictly decreasing.
\end{lemma}
    
\begin{proof}
    From Lemma~\ref{lm:xu_mfg_monotone}, $u_{\lambda,Z}$ is strictly decreasing with respect to $Z$. Now suppose $Z_1 < Z_2$ then
    \[ 0 = I_1(\lambda(Z_2),Z_2) = I_1(\lambda(Z_1),Z_1) > I_1(\lambda(Z_1),Z_2) \, . \]
    Therefore, since $I_1$ is strictly decreasing in $\lambda$, $\lambda(Z_2) < \lambda(Z_1)$ so $\lambda(Z)$ is strictly decreasing with respect to $Z$.
\end{proof}
    
\begin{lemma}
    Assume $h$ is strictly increasing with respect to $m$. Then the function $\lambda(Z)$ is continuous.
\end{lemma}
    
\begin{proof}
    We will prove $\lambda(Z)$ is sequentially continuous. Let $Z_n$ be a sequence in $(0,\infty)$ that converges to $Z \in (0,\infty)$. We consider two sequences: $Z_{n}^{(i)}$ for $i = 1,2$, which we use to sandwich our original sequence. We choose these sequences as follows
    \begin{enumerate}
        \item $Z_n^{(1)} = \inf_{j \geq n} Z_j$
        \item $Z_n^{(2)} = \sup_{j \geq n} Z_j$
    \end{enumerate}
    
    Now, $Z_n^{(1)},Z_n^{(2)} \to Z$ and are increasing and decreasing sequences respectively. Furthermore, there exists $\underaccent{\bar}{Z},\bar{Z}$ such that $Z_n^{(1)},Z_n^{(2)} \in [\underaccent{\bar}{Z},\bar{Z}]$ for every $n \in \mathbb{N}$. So, since $\lambda(Z)$ is decreasing, $\lambda(Z_n^{(1)}),\lambda(Z_n^{(2)}) \in [\lambda(\bar{Z}),\lambda(\underaccent{\bar}{Z})]$ and are decreasing and increasing respectively. Therefore, $\lambda(Z_n^{(1)}) \to \lambda^{(1)}$ and $\lambda(Z_n^{(2)}) \to \lambda^{(2)}$. Using continuity of $I_1$ we get for $i = 1,2$:
    \[ 0 = \lim_{n \to \infty} I_1 \left( \lambda(Z_n^{(i)}),Z_n^{(i)} \right) =  I_1 \left( \lambda^{(i)},Z \right) \, . \]
    Hence by definition $\lambda^{(1)}  = \lambda^{(2)} = \lambda(Z)$. Since $Z_n^{(i)}$ bound $Z_n$, then $\lambda(Z_n^{(i)})$ bound $\lambda(Z_n)$. Hence $\lambda(Z_n) \to \lambda(Z)$.
\end{proof}

\begin{proposition} \label{prop:xu_mfg_Z}
    Define $I_2:(0,\infty) \to (0,\infty)$ by
    \begin{equation}
        I_2(Z) = \int_{\Omega} \frac{1}{Z} e^{- \frac{2}{\sigma^2} u_{\lambda(Z),Z}}~dx \, .
    \end{equation}
    Assume $h$ is strictly increasing with respect to $m$. Then there exists a unique $Z^* \in (0,\infty)$ such that $I_2(Z^*) = 1$.
\end{proposition}
    
\begin{proof}
    Similar to the proof of Proposition~\ref{prop:xu_mfg_lambda}, we prove this proposition using the intermediate value theorem. Again there are three parts we have to prove
     \begin{enumerate}
        \item There exists $Z_1 \leq Z_2 \in (0, \infty)$ such that $I_2(Z_1) \geq 1$ and $I_2(Z_2) \leq 1$.
        \item $I_2(Z)$ is continuous with respect to $Z$ for all $Z \in [Z_1,Z_2]$.
        \item $I_2(Z)$ is strictly decreasing with respect to $Z$.
    \end{enumerate}
    Steps (1) and (2) prove existence via the intermediate value theorem, step (3) proves uniqueness.
    
    Step (1): From assumption~\ref{a:mfg_mto0}, we can find $Z_2 \geq |\Omega|$ such that 
    \begin{equation} \label{eq:xu_mfg_Zbound}
        \sup_{x \in \Omega} h \left( x,\frac{1}{Z_2} \right) \leq \inf_{x \in \Omega} h \left(x, \frac{1}{|\Omega|} \right) \, .
    \end{equation}
    Then, since $\lambda(Z_2) \leq \sup_{x \in \Omega} h \left( x,\frac{1}{Z_2} \right)$ (from~\eqref{eq:xu_mfg_lambdabound}), it follows that
    \[ u_{\lambda(Z_2),Z_2} \geq u_{\sup_{x \in \Omega} h \left( x,\frac{1}{Z_2} \right), Z_2} \geq \underaccent{\bar}{u}_{\sup_{x \in \Omega} h \left( x,\frac{1}{Z_2} \right), Z_2} = \min \left( - \frac{\sigma^2}{2} \log Z_2 M,0 \right) \, , \]
    where $M$ satisfies $h(x,M) \geq \sup_{x \in \Omega} h \left( x,\frac{1}{Z_2} \right)$ for all $x$ (from the proof of Lemma~\ref{lm:xu_mfg_constbound}). But from~\eqref{eq:xu_mfg_Zbound}, this is clearly satisfied by $M = \frac{1}{|\Omega|}$, and in this case $\min \left( - \frac{\sigma^2}{2} \log Z_2 M,0 \right) = - \frac{\sigma^2}{2} \log \frac{Z_2}{|\Omega|}$. Thus
    \[ I_2(Z_2) \leq \int_{\Omega} \frac{1}{Z_2} e^{- \frac{2}{\sigma^2} \left( - \frac{\sigma^2}{2} \log \frac{Z_2}{|\Omega|} \right)}~dx  = \int_{\Omega} \frac{1}{|\Omega|}~dx = 1 \, . \]
    A similar procedure works to find $Z_1$, in which case $Z_1$ satisfies $Z_1 \leq |\Omega|$ and $\inf_{x \in \Omega} h \left( x, \frac{1}{Z_1} \right) \geq \sup_{x \in \Omega} h \left( x, \frac{1}{|\Omega|} \right)$.
    
    Step (2): Take $Z_1 \leq Z_2$ as in Step (1). Then for every $Z \in [Z_1,Z_2]$ there exists $C_1,C_2 \in \R$ such that ( by~\eqref{eq:xu_mfg_lambdabound})
    \[ C_2 = \inf_{x \in \Omega} h \left(x, \frac{1}{Z_2} \right) \leq \lambda(Z_2) \leq \lambda(Z) \leq \lambda(Z_1) \leq  \sup_{x \in \Omega} h \left(x, \frac{1}{Z_1} \right) = C_1 \, . \]
    So $\underaccent{\bar}{u}_{C_1,Z_2} \leq u_{\lambda(Z),Z} \leq \bar{u}_{C_2,Z_1}$ for every $Z \in [Z_1,Z_2]$. So we can use the dominated convergence theorem along with continuity of $u_{\lambda,Z}$ with respect $(\lambda,Z)$ and continuity of $\lambda(Z)$ with respect to $Z$ to show $I_2(Z)$ is continuous.
    
    Step (3): Take $\underaccent{\bar}{Z} < \bar{Z}$ then there exists $a \in \Omega$ such that
    \[ u_{\lambda(\underaccent{\bar}{Z}),\bar{Z}}(a) < u_{\lambda(\bar{Z}),\bar{Z}}(a) \, . \]
    Therefore, at $a \in \Omega$:
    \[ \frac{1}{\underaccent{\bar}{Z}} e^{- \frac{2}{\sigma^2} u_{\lambda(\underaccent{\bar}{Z}),\underaccent{\bar}{Z}}} = \frac{1}{\bar{Z}} e^{- \frac{2}{\sigma^2} u_{\lambda(\underaccent{\bar}{Z}),\bar{Z}}} > \frac{1}{\bar{Z}} e^{- \frac{2}{\sigma^2} u_{\lambda(\bar{Z}),\bar{Z}}} \, . \]
    So $I_2(\underaccent{\bar}{Z}) > I_2(\bar{Z})$ because of the continuity of $u_{\lambda,Z}$. This proves $I_2$ is strictly decreasing.
\end{proof}
    
\begin{proof}[End of proof of Theorem~\ref{thm:xu_mfg1}]
    First let's assume $h$ is a strictly increasing function in $m$. Then we can choose the unique $Z^* \in (0,\infty)$ such that $I_2(Z^*) = 1$. Then clearly the triple ${\left( u_{\lambda(Z^*),Z^*}, \lambda(Z^*), Z^* \right)}$ is a solution to the system~\eqref{eq:xu_mfg_sys}. Furthermore, suppose $(u',\lambda',Z')$ is also a solution of~\eqref{eq:xu_mfg_sys}. But this implies that $u'$ satisfies~\eqref{eq:xu_mfg_bvp1}--\eqref{eq:xu_mfg_bvp3}, so $u' = u_{\lambda',Z'}$ from uniqueness proven in Proposition~\ref{prop:xu_mfg_1}. Then $u_{\lambda',Z'}$ also solves~\eqref{eq:xu_mfg_bvp5}, so by uniqueness proven in Proposition~\ref{prop:xu_mfg_lambda} we can show $\lambda' = \lambda(Z')$. Finally we now have $u' = u_{\lambda(Z'),Z'}$ meets the integral constraint~\eqref{eq:xu_mfg_bvp4}. So from uniqueness proven in Proposition~\ref{prop:xu_mfg_Z} we have $Z' = Z^*$. Therefore ${(u',\lambda',Z') = \left( u_{\lambda(Z^*),Z^*}, \lambda(Z^*), Z^* \right)}$. Hence the unique solution to the MFG problem is given by $\left( m_{Z^*},u_{\lambda(Z^*),Z^*},\lambda(Z^*) \right)$, where $m_{Z^*}$ is defined by $m_{Z^*} = \frac{1}{Z^*} e^{- \frac{2}{\sigma^2}u_{z^*}}$.
    
    Now let's assume $h$ is an increasing function in $m$ and define $h_{\epsilon}(x,m)$ by
    \[ h_{\epsilon}(x,m) = h(x,m) + \epsilon \log \left(|\Omega| m \right) \, . \]
    Then for every $\epsilon \in (0,1]$, $h_{\epsilon}$ is a strictly increasing function of $m$. Furthermore $h_{\epsilon}$ still satisfies assumptions~\ref{a:mfg_hreg}--\ref{a:mfg_mtoinf}. Therefore there exists a unique solution $(u_{\epsilon},\lambda_{\epsilon},Z_{\epsilon})$ to the MFG system~\eqref{eq:xu_mfg_sys}. From Proposition~\ref{prop:xu_mfg_Z}, $Z_{\epsilon} \in [Z^1_{\epsilon},Z^2_{\epsilon}]$ for some $Z^1_{\epsilon},Z^2_{\epsilon} \in (0,\infty)$ such that
    \[ \begin{aligned}
        0 < Z^1_{\epsilon} \leq &|\Omega| \leq Z^2_{\epsilon} < \infty \\
        \sup_{x \in \Omega} h_{\epsilon} \left(x,\frac{1}{Z^2_{\epsilon}}\right) &\leq \inf_{x \in \Omega} h_{\epsilon} \left(x,\frac{1}{|\Omega|}\right) \\
        \inf_{x \in \Omega} h_{\epsilon} \left(x,\frac{1}{Z^1_{\epsilon}}\right) &\geq \sup_{x \in \Omega} h_{\epsilon} \left(x,\frac{1}{|\Omega|}\right) \, .
    \end{aligned} \]
    But by the definition of $h_{\epsilon}$ we have $h_{\epsilon} \left(x,\frac{1}{Z^2_{\epsilon}}\right) \leq h \left(x,\frac{1}{Z^2_{\epsilon}}\right)$ and $h_{\epsilon} \left(x,\frac{1}{|\Omega|}\right) = h \left(x,\frac{1}{|\Omega|}\right)$, and a similar inequality holds for $Z^1_{\epsilon}$ . So we can find $Z^1 \in (0,|\Omega|]$ and $Z^2 \in [|\Omega|,\infty)$ independent of $\epsilon$ such that $Z_{\epsilon} \in [Z^1,Z^2]$ for every $\epsilon \in (0,1]$. Now, from Lemma~\ref{lm:xu_lambda_monotone} 
    \[ \lambda_{\epsilon} = \lambda(Z_{\epsilon}) \in \left[\lambda(Z^2),\lambda(Z^1)\right] \, . \]
    So take a sequence $\epsilon_n$ such that $\lim_{n \to \infty} \epsilon_n = 0$. Then, since $u_{\epsilon} \in C^{2,\tau}\left(\bar{\Omega}\right)$, which is compactly embedded in $C^2\left(\Omega\right) \cap C^1\left(\bar{\Omega}\right)$, there exists a subsequence also denoted by $n$ such that $u_{\epsilon_n} \to u_0$ with convergence in $C^2\left(\Omega\right) \cap C^1\left(\bar{\Omega}\right)$, $Z_{\epsilon_n} \to Z_0 \in [Z^1,Z^2]$, and $\lambda_{\epsilon_n} \to \lambda_0 \in \left[\lambda(Z^2),\lambda(Z^1)\right]$. So we find, by taking limits
    \[ - \frac{\sigma^2}{2} \nabla^2 u_0 + \frac{|\nabla u_0|^2}{2} - h\left(x,\frac{1}{Z_0}e^{- \frac{2}{\sigma^2}u_0}\right) + \lambda_0 = 0 \, .\]
    Similarly we can show $(u_0,\lambda_0,Z_0)$ satisfy~\eqref{eq:xu_mfg_sys}. So we have proven existence of solutions for increasing $h$. 
    
    For uniqueness, let's assume $(u_1,\lambda_1,Z_1)$ and $(u_2,\lambda_2,Z_2)$ are both solutions of~\eqref{eq:xu_mfg_sys} and $\lambda_1 \leq \lambda_2$. Define $u = u_2 - \frac{\sigma^2}{2} \log\left(\frac{Z_1}{Z_2}\right)$, then $u$ satisfies
    \[ - \frac{\sigma^2}{2} \nabla^2 u + \frac{|\nabla u|^2}{2} - h\left(x,\frac{1}{Z_1}e^{- \frac{2}{\sigma^2}u}\right) + \lambda_1 = 0 \, . \] 
    Now define $v = u - u_2$ and suppose there exists $x \in \bar{\Omega}$ such that $v(x) > 0$. Then $v$ has a maximum at $x^*$ and $v(x^*) > 0$. By Hopf's lemma $x^* \in \Omega$ and by the maximum principle $v$ is constant in the set $\Omega_+ = \{x \in \Omega : v(x) \geq 0\}$ (see the proof of Proposition~\ref{prop:xu_mfg_1} for the details of such an argument). By assumption $x^* \in \Omega_+$ and $v(x^*) > 0$, so $\Omega_+ = \Omega$ by continuity of $v$. Hence $v$ is constant in $\Omega$ and $v > 0$. However, from the integral constraint~\eqref{eq:xu_mfg_bvp4} we obtain
    \begin{equation} \label{eq:xu_mfg_proof1}
        0 = \int_{\Omega} \frac{1}{Z_1} e^{- \frac{2}{\sigma^2} u_1} \left(1 - e^{- \frac{2}{\sigma^2} v}\right)~dx = |\Omega| \left(1 - e^{- \frac{2}{\sigma^2} v}\right) \, ,
    \end{equation}
    since $1 - e^{- \frac{2}{\sigma^2} v}$ is constant. So $v = 0$, contradicting the assumption $v(x^*) > 0$. Therefore $v(x) \leq 0$ for all $x \in \Omega$. This implies that $1 - e^{- \frac{2}{\sigma^2} v} \leq 0$, and subsequently that
    \[0 = \int_{\Omega} \frac{1}{Z_1} e^{- \frac{2}{\sigma^2} u_1} \left(1 - e^{- \frac{2}{\sigma^2} v}\right)~dx \leq 0 \, ,\]
    with equality if and only if $v = 0$. Therefore $u_2 = u$, which implies (using the integral constraint~\eqref{eq:xu_mfg_bvp5}) that $Z_1 = Z_2$, and subsequently that $u_1 = u_2$. Finally, by subtracting the PDE~\eqref{eq:xu_mfg_bvp2} satisfied by $u_1$ from the one satisfied by $u_2$ we find $\lambda_2 = \lambda_1$. Therefore solutions are unique.
\end{proof}

\section{Quadratic Potential} \label{sec:quad potential}

\noindent In this section we consider a specific example with quadratic potential and a logarithmic congestion term, $h(x,m) = \beta x^2 + \log m$ for some constant $\beta \geq 0$ on the real line. This problem has been studied extensively in~\cite{Gomes2016a} and~\cite{Gueant2009} and admits explicit solutions. This allows us to compare the solutions of the BRS and the MFG. Note that we do not impose any boundary conditions or integral constraints on $u$, since we consider the model on $\R$ rather than on a bounded domain. Therefore it doesn't fit directly into the framework for existence and uniqueness proven in the previous section. It is however, one of the few illustrative examples, where explicit solutions are known. This allows us to make an analytical comparison of the two models and use the solution to validate the proposed numerical methods.

\subsection{The MFG}

The stationary MFG model studied in~\cite{Gomes2016a} and~\cite{Gueant2009}, with the integral constraints used in this paper, is given by:
\begin{subequations}\label{eq:quad_stat_mfg1}
\begin{align} 
    \frac{\sigma^2}{2} \partial_{xx}^2 m + \partial_x \left( m \partial_x u \right) &= 0 \, , \quad x \in \R \, , \\ \label{eq:quad_stat_mfg2}
    - \frac{|\partial_x u|^2}{2} + \log m + \beta x^2 + \frac{\sigma^2}{2} \partial_{xx}^2 u + \lambda &= 0 \, , \quad x \in \R \, , \\
    \int_{\R} m~dx &= 1 \, .
\end{align}
\end{subequations}
where $\lambda \in (-\infty, \infty)$ is a constant to be found as part of the solution, and $\sigma, \beta \geq 0$ are given parameters.
\begin{proposition}
    A solution to the stationary MFG system~\eqref{eq:quad_stat_mfg1} exists and has an explicit form 
    \begin{subequations} \label{eq:quad_stat_mfg_sol}
        \begin{align}
            & m(x) = \left( \frac{a}{\pi} \right)^{1/2} e^{- a x^2} \\
            & u(x) = b x^2 \\
            & \lambda = \log \left( \frac{\pi}{a} \right) - \sigma^2 b\, ,
        \end{align}
    \end{subequations}
    where the constants $a,b,c \geq 0$ are given by
    \[ a = \beta,~ b = 0 \, , \]
    if $\sigma = 0$, or
    \[ a = \frac{-1 + \left( 1 + 2 \sigma^4 \beta \right)^{1/2}}{\sigma^4},~b = \frac{-1 + \left( 1 + 2 \sigma^4 \beta \right)^{1/2}}{2 \sigma^2} \, , \]
    if $\sigma > 0$.
\end{proposition}
The proof is straight-forward using substitution.

\subsection{The BRS}
Next we consider the respective stationary BRS model. It is given by
\begin{subequations} \label{eq:quad_stat_brs}
    \begin{align}
        \partial_x \left( m \partial_x (\log m + \beta x^2) \right) + \frac{\sigma^2}{2} \partial_{xx}^2 m &= 0 \, , \quad x \in \R \\
        \int_{\R} m~dx &= 1 \, .
    \end{align}
\end{subequations}
\begin{proposition}
    The solution to the stationary BRS equation~\eqref{eq:quad_stat_brs} is given by
    \[ m(x) = \left( \frac{2 \beta}{(2 + \sigma^2) \pi} \right)^{1/2} e^{- \frac{2 \beta}{(2 + \sigma^2)} x^2} \, . \]
\end{proposition}
Again the claim follows from substitution. 

\begin{figure}[h!]
    \begin{subfigure}{0.49 \textwidth}
	    \centering
	    \includegraphics[width=\linewidth]{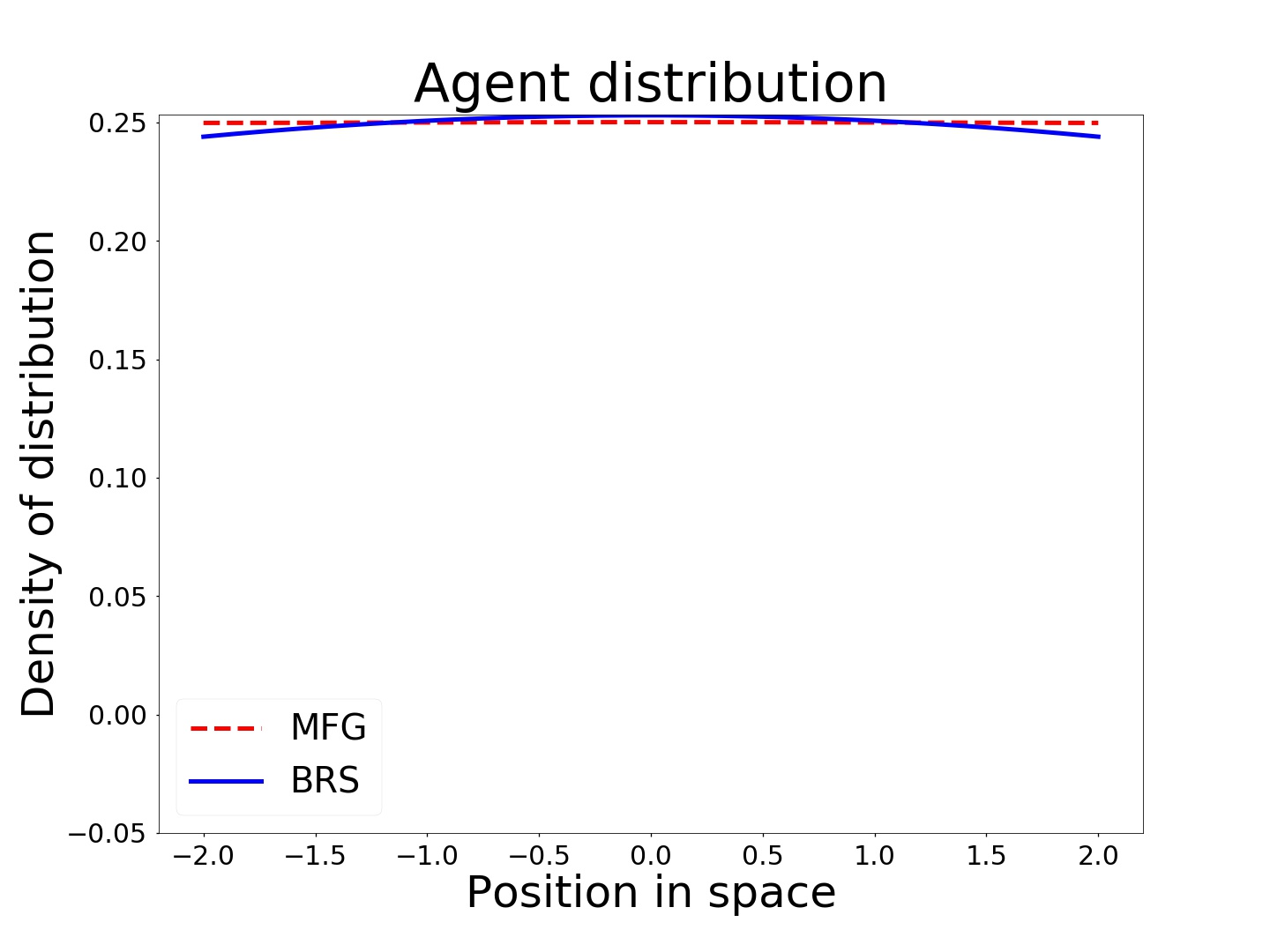}
	    \caption{$\beta = 0.1$, $\frac{\sigma^2}{2} = 10$}
    \end{subfigure}
    \begin{subfigure}{0.5 \textwidth}
	    \centering
	    \includegraphics[width=\linewidth]{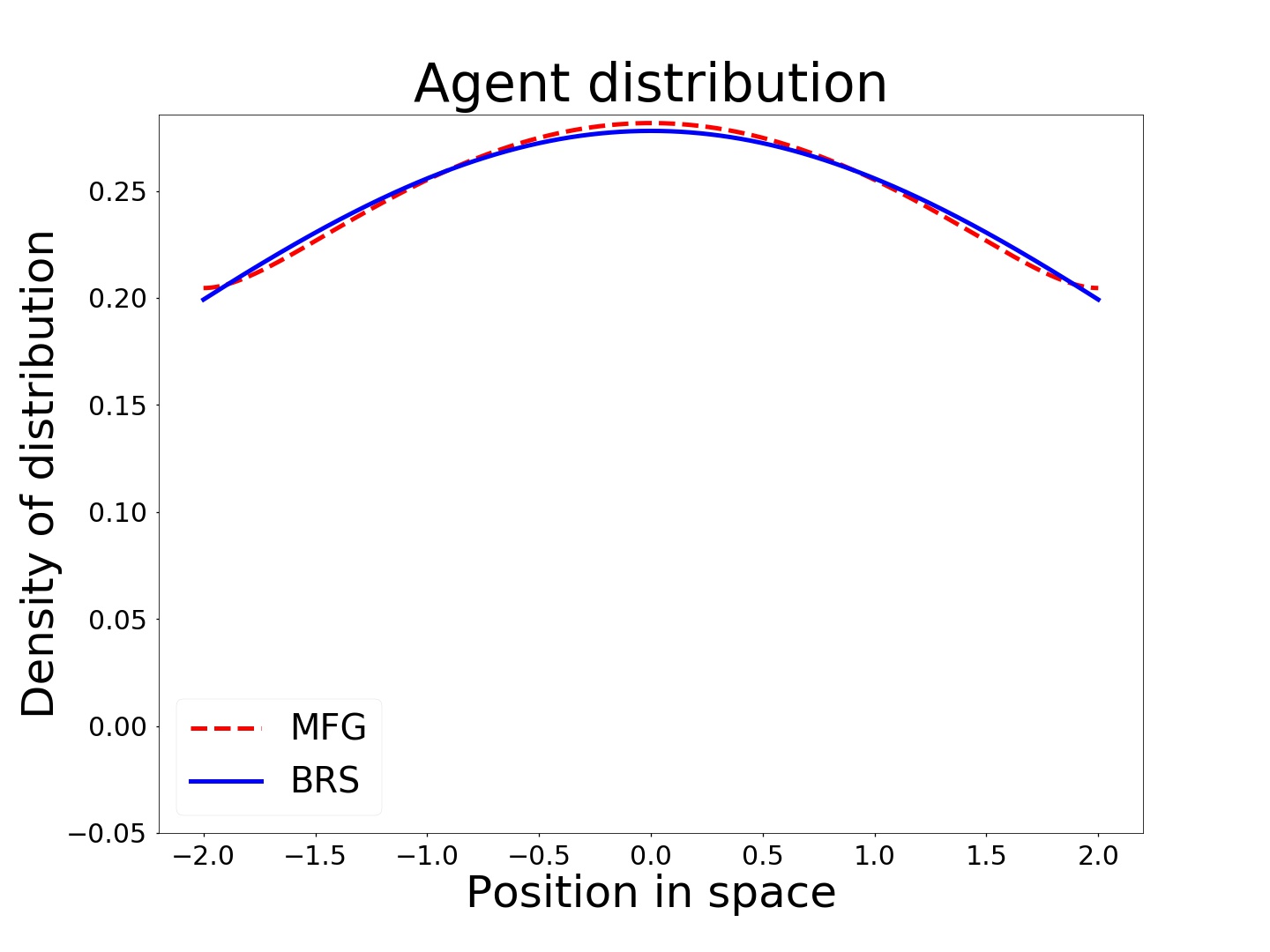}
	    \caption{$\beta = 0.1$, $\frac{\sigma^2}{2} = 0.2$}
    \end{subfigure}
    
    \bigskip
    
    \begin{subfigure}{0.49 \textwidth}
	    \centering
	    \includegraphics[width=\linewidth]{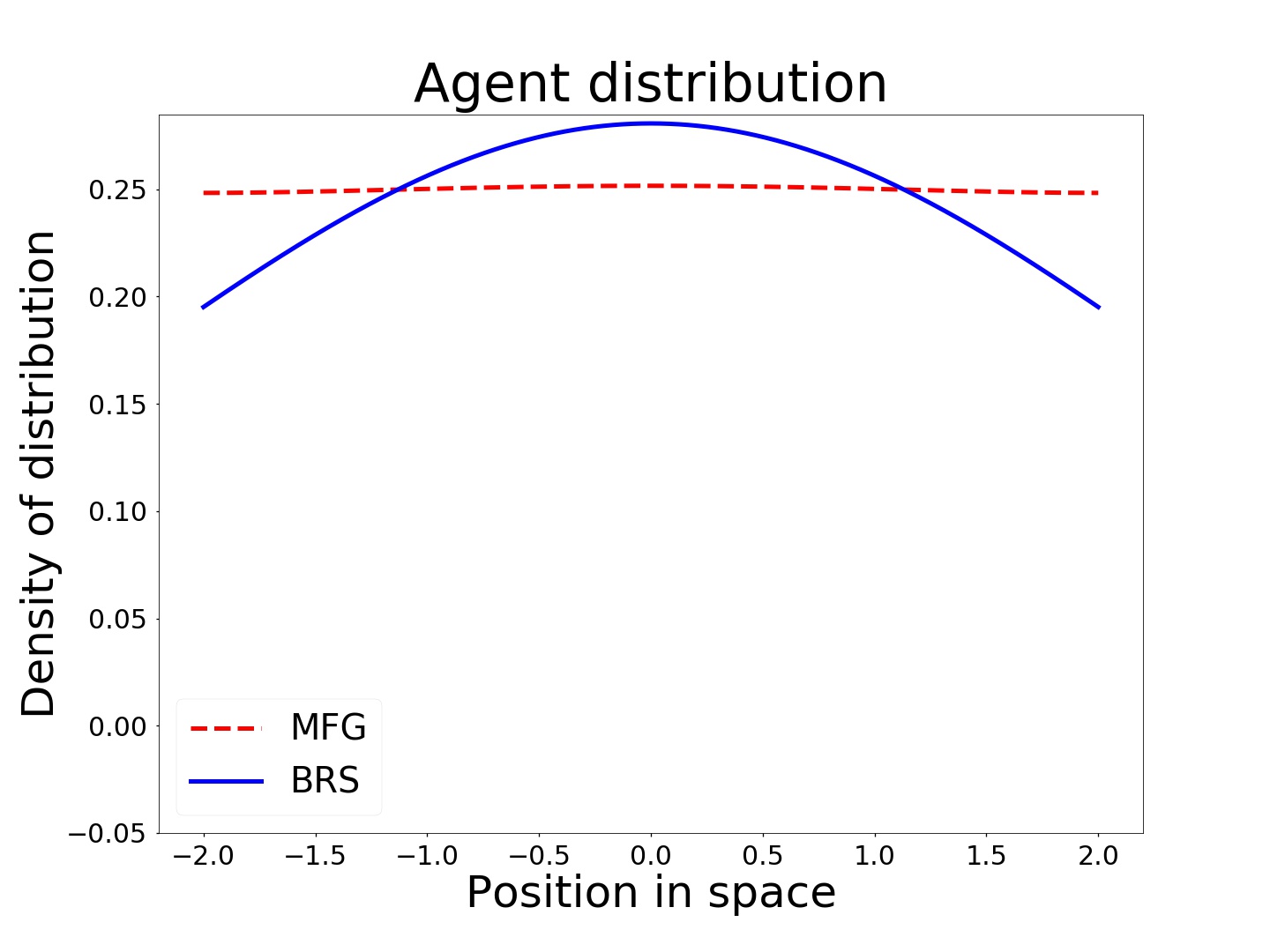}
	    \caption{$\beta = 1$, $\frac{\sigma^2}{2} = 10$}
    \end{subfigure}
    \begin{subfigure}{0.49 \textwidth}
	    \centering
	    \includegraphics[width=\linewidth]{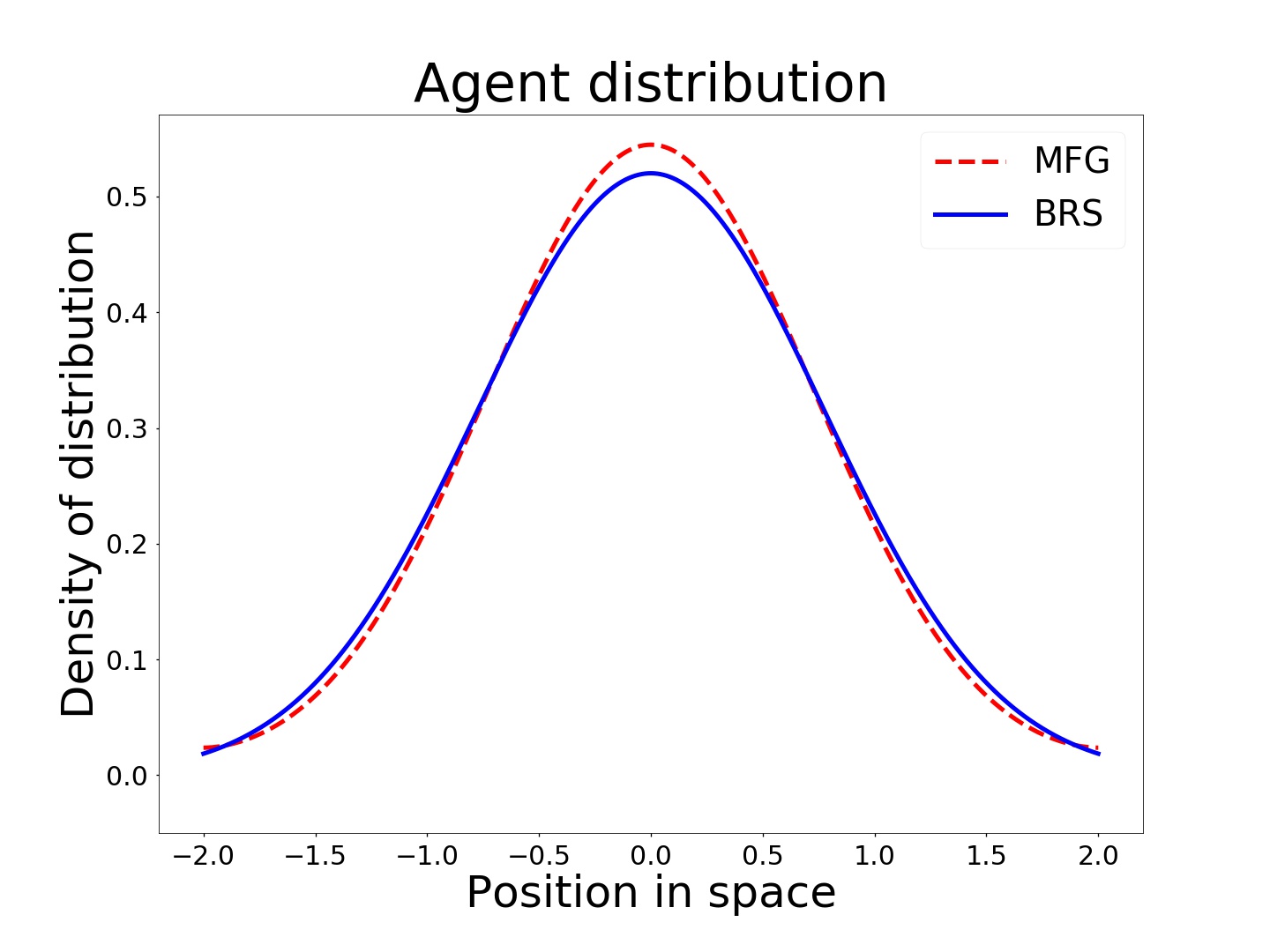}
	    \caption{$\beta = 1$, $\frac{\sigma^2}{2} = 0.2$}
    \end{subfigure}
    
     \bigskip
    
    \begin{subfigure}{0.49 \textwidth}
	    \centering
	    \includegraphics[width=\linewidth]{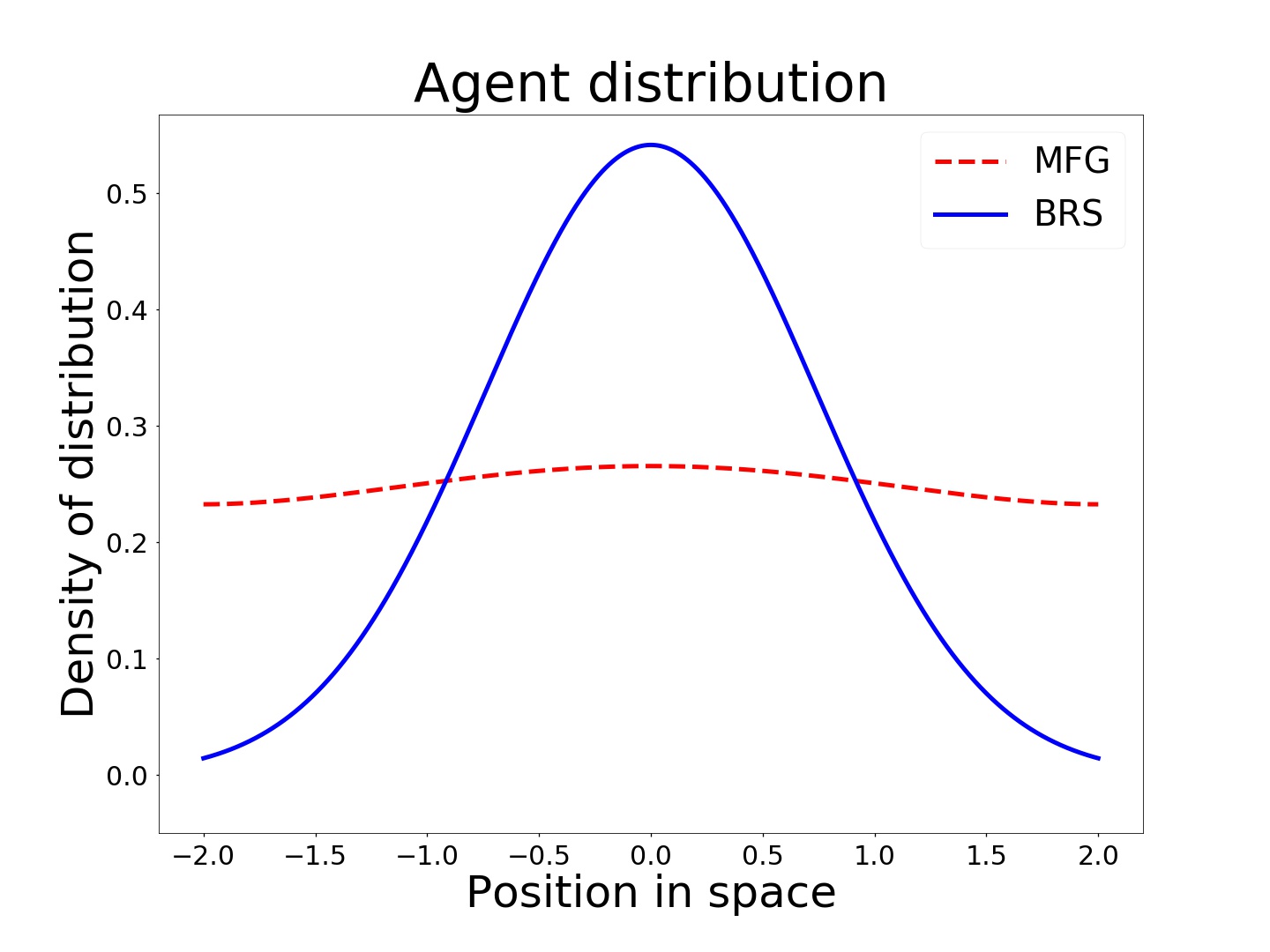}
	    \caption{$\beta = 10$, $\frac{\sigma^2}{2} = 10$}
    \end{subfigure}
    \begin{subfigure}{0.49 \textwidth}
	    \centering
	    \includegraphics[width=\linewidth]{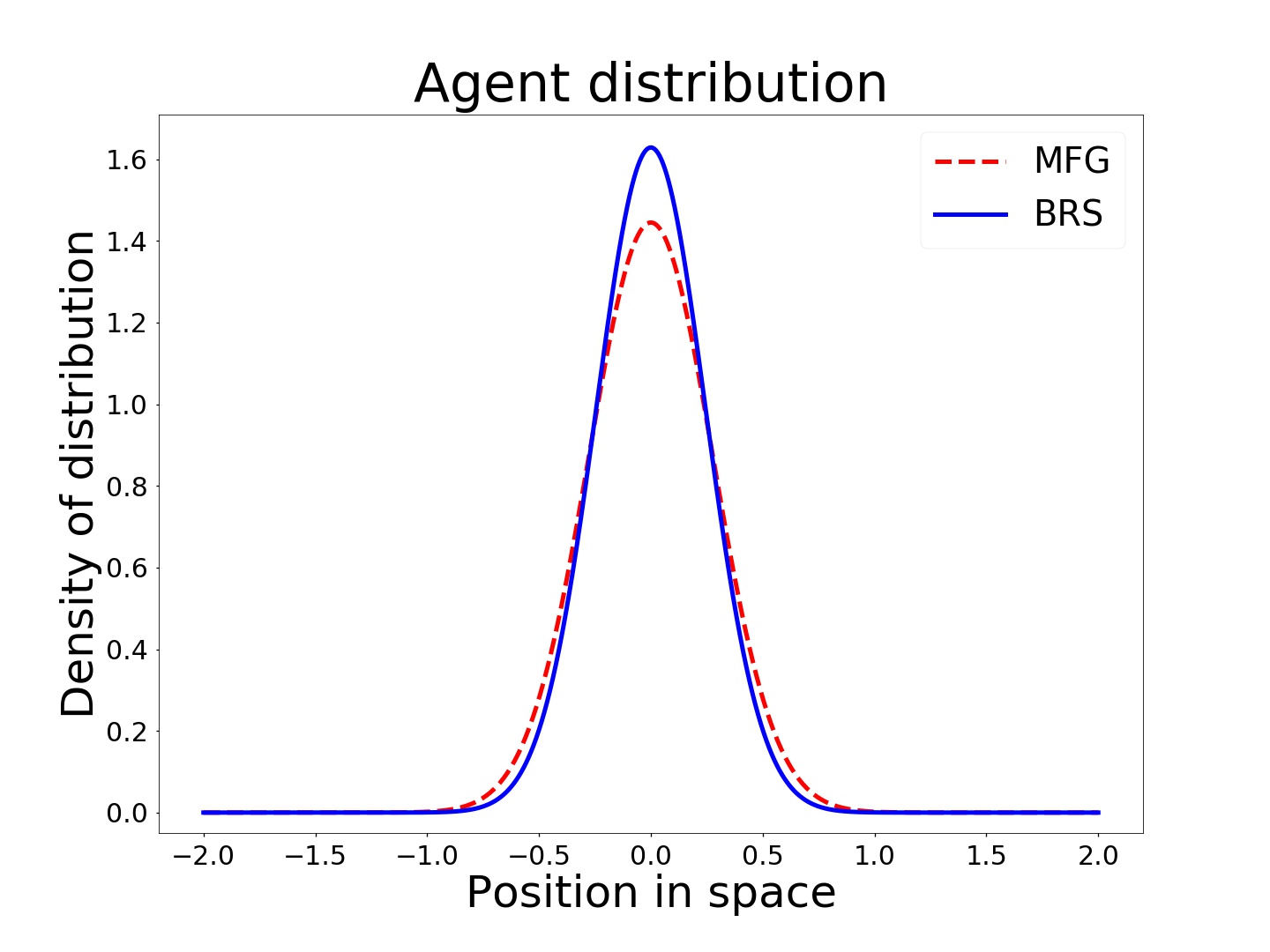}
	    \caption{$\beta = 10$, $\frac{\sigma^2}{2} = 0.2$}
    \end{subfigure}
    \caption{Simulations of BRS and MFG with quadratic potential and logarithmic congestion}
    \label{fig:quad-log}
\end{figure}

\subsection{Comparison}
\begin{proposition}
    For $\sigma > 0$, the stationary distributions of the MFG system~\eqref{eq:quad_stat_mfg1} and the BRS~\eqref{eq:quad_stat_brs} are given by normal distributions with mean $0$ and variances $a_1$ and $a_2$ respectively, where
    \[ \begin{aligned}
        a_1 & = \frac{\sigma^4}{-2 + 2(1 + 2 \sigma^4 \beta)^{1/2}} \\
        a_2 & = \frac{2 + \sigma^2}{4 \beta} \, .
    \end{aligned} \]
    Then for fixed $\beta \geq 0$
    \begin{subequations}
        \begin{align}
            \lim_{\sigma^2 \to 0} \frac{a_2}{a_1} &= 1 \label{eq:quad_lim1} \\
            \lim_{\sigma^2 \to \infty} \frac{a_2}{a_1} &= \frac{1}{(2 \beta)^{1/2}} \, . \label{eq:quad_lim2}
        \end{align}
    \end{subequations}
    While for fixed $\sigma > 0$
    \begin{subequations}
        \begin{align}
            \lim_{\beta \to 0} \frac{a_2}{a_1} &= 1 + \frac{\sigma^2}{2} \label{eq:quad_lim3} \\
            \lim_{\beta \to \infty} (2 \beta)^{1/2} \frac{a_2}{a_1} &= \frac{2 + \sigma^2}{\sigma^2} \, . \label{eq:quad_lim4}
        \end{align}
    \end{subequations}
\end{proposition}

\begin{proof}
    The first part of this proof is trivial from the previous propositions. Now 
    \[ \frac{a_2}{a_1} = \frac{(2 + \sigma^2) \left( (1 + 2 \sigma^4 \beta)^{1/2} - 1 \right)}{2 \sigma^4 \beta} \, . \]
    Using a Taylor expansion of $(1 + x)^{1/2}$ around $x = 0$ gives behaviour for small $\sigma^2$ i.e.
   \[ \frac{a_2}{a_1} = \frac{(2 + \sigma^2) (\sigma^4 \beta + o(\sigma^4))}{2 \sigma^4 \beta} \, . \]
    Hence $\lim_{\sigma^2 \to 0} \frac{a_2}{a_1} = \lim_{\sigma^2 \to 0} \frac{2 + \sigma^2}{2} = 1$. The other limit can be simply calculated
    \[ \begin{aligned}
            \lim_{\sigma^2 \to \infty} \frac{a_2}{a_1} & = \lim_{\sigma^2 \to \infty} \frac{(2 + \sigma^2) (1 + 2 \sigma^4 \beta)^{1/2}}{2 \sigma^4 \beta} = \lim_{\sigma^2 \to \infty} \frac{(2 + \sigma^2) (2 \beta)^{1/2} \sigma^2}{2 \sigma^4 \beta} = \frac{1}{(2 \beta)^{1/2}} \, .
        \end{aligned} \]
    The limits as $\beta \to 0,\infty$ for fixed $\sigma$, follows from straight forward calculations.
\end{proof}

This result is an important first glimpse at how the behaviour of the BRS and MFG may vary, as well as the importance certain parameters play in the difference. The limit in~\eqref{eq:quad_lim1} shows that the existence of noise is vital to see any difference between the two models. However, as soon as there is noise, its effect on the relative difference plays a less important role than the strength of the quadratic potential, this can be seen in~\eqref{eq:quad_lim2} and~\eqref{eq:quad_lim4}. Specifically the limit~\eqref{eq:quad_lim4} shows that the relative difference between the variances of the two distributions grows like $\beta^{\frac{1}{2}}$, which means the BRS distribution reacts much more rapidly with changes to the potential strength than the MFG. This suggests that the MFG is more affected by congestion or is a more congestion-averse model than the BRS one. 

At a conceptual level this agrees with the formulation of the MFG and BRS systems. The agents in the BRS are acting myopically, only reacting to the situation as it currently exists, which isn't the case in the MFG. As a result the BRS agents don't `see' the future congestion that will result from their behaviour and hence they move towards the minimum of $\beta x^2$ more rapidly than the MFG agents who do see the future cost of the congestion that results from their behaviour. Therefore, thinking of the stationary solutions as the long time, time-averaged behaviour of the models then the stationary BRS will result in a distribution that appears to take into account the congestion less than the MFG and hence one with a smaller variance. This expectation is confirmed by the result~\eqref{eq:quad_lim4} and it in fact quantifies the extent to which the BRS ignores the congestion compared with the MFG.

For this model we have run a variety of simulations --- both to confirm our numerical methods (see Section~\ref{sec:numerical_sim} for methods) and to visualise how the parameters affect the distributions. Figure~\ref{fig:quad-log} shows the results of these simulations on a bounded domain for a variety of parameter choices. Although the formulation on a bounded domain is slightly different than the one in this section, the same behaviour can be seen. For small $\sigma$ the difference between the models doesn't change much as $\beta$ increases, while for large values of $\sigma$ the BRS model is much more dramatically affected by changes to $\beta$. In both cases the BRS and MFG are more closely aligned when $\beta$ is small. 

\section{Simulations}\label{sec:numerical_sim}
\noindent We conclude with presenting various computational experiments, which illustrate the difference between solutions to the BRS and MFG for different choices running costs and potentials. 

\subsection{Solving the stationary BRS and MFG}
Solutions to the stationary BRS~\eqref{eq:brssystem} can be computed by finding the zeros of the function $G_{Z,x}(m)$ at every discrete grid point $x$ on a grid given $Z$. To compute the roots of $G_{Z,x}$ at every grid point $x$ we use a Newton-Raphson method. Then, having found $m_Z(x)$ for a particular value of $Z$, we can differentiate the implicit formula $m_Z = \frac{1}{z} e^{- \frac{2}{\sigma^2} h(x,m_Z)}$ with respect to $Z$ and use a Newton-Raphson method to find $Z$ such that $\Phi(Z) = \int_{\Omega}m_Z~dx = 1$. In practice this means iterating between the two Newton-Raphson methods: first finding $m_{Z_n}$, then computing $Z_{n+1}$, then recomputing $m_{Z_{n+1}}$ and repeating until convergence.\\
The solution to the stationary MFG~\eqref{eq:xu_mfg} are found using an iterative procedure. Given an admissible initial iterate $m^l$, $l=0$ we solve the HJB equation~\eqref{eq:xu_mfg_pde2} to obtain $u^l$ and $\lambda^l$. Note that we include the constraint~\eqref{eq:xu_mfg_ic2} via a Lagrange multiplier. In the final step of the iteration we update the distribution of agents by solving the FPE~\eqref{eq:xu_mfg_pde1} using $u^{l+1}$ to obtain $m^{l+1}$. This procedure is repeated until convergence. Note that we sometimes perform a damped update
\begin{align*}
u^l = \omega u^{l-1} + (1-\omega) v^l   \text{ and } m^l = m^{l-1} + (1-\omega) q^l 
\end{align*}
where $v^l,q^l$ are the undamped solutions of each iteration process. This damping helps to ensure convergence. Solutions to the HJB and the FPE are obtained by using an $H^1$ conforming finite element discretisation.

\begin{figure}[b]
     \centering
    \includegraphics[width=0.7\textwidth]{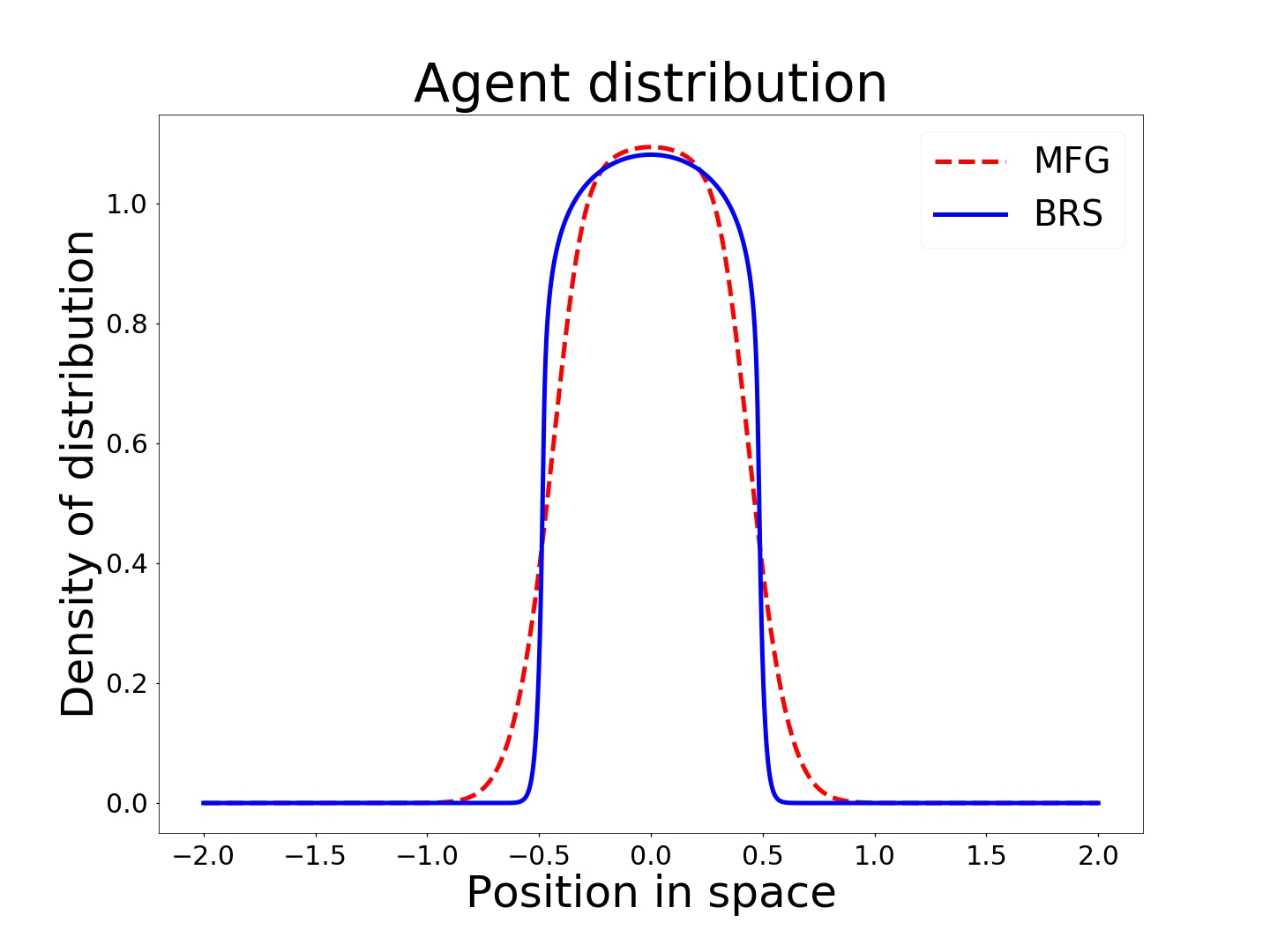}
    \caption{Simulations of BRS and MFG with $h(x,m) = m^{10} + 10 x^2$}
    \label{fig:quad-power}
\end{figure}

We noticed that in many of the simulations performed the ``cost'' associated to the MFG was higher than the ``cost'' associated with the BRS. At first this sounds counter-intuitive as the BRS (in the dynamic case) is a sub-optimal approximation of the MFG. However, as explained in Section~\ref{sec:stat_prob}, the ``cost'' function $u$ is actually the long-time average difference between the cost and the space-average cost, whereas the stationary BRS cost is the equilibrium of the competitive minimisation of \eqref{eq:Estat}. Therefore the MFG cost will always be centred around 0 while the BRS cost could be above or below it. As a result, it was not clear that comparing the ``costs'' of the two models is especially useful and hence we have solely focussed on comparing the distribution of agents.

\begin{figure}[t]
    \begin{subfigure}{0.49 \textwidth}
	    \centering
	    \includegraphics[width=\linewidth]{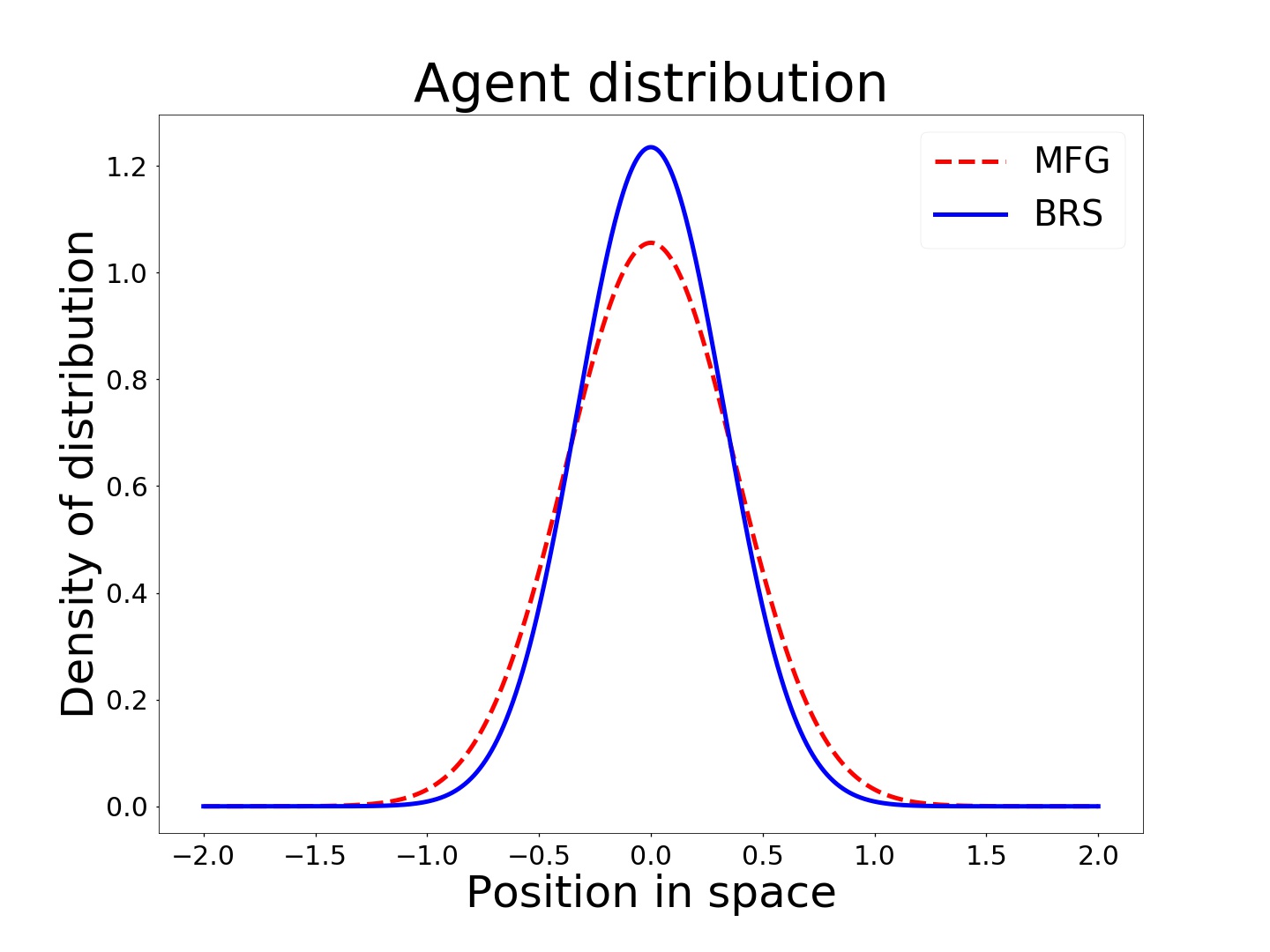}
	    \caption{$m_{\max} = 10$}
	    \label{fig:quad-max-10}
    \end{subfigure}
    \begin{subfigure}{0.49 \textwidth}
	    \centering
	    \includegraphics[width=\linewidth]{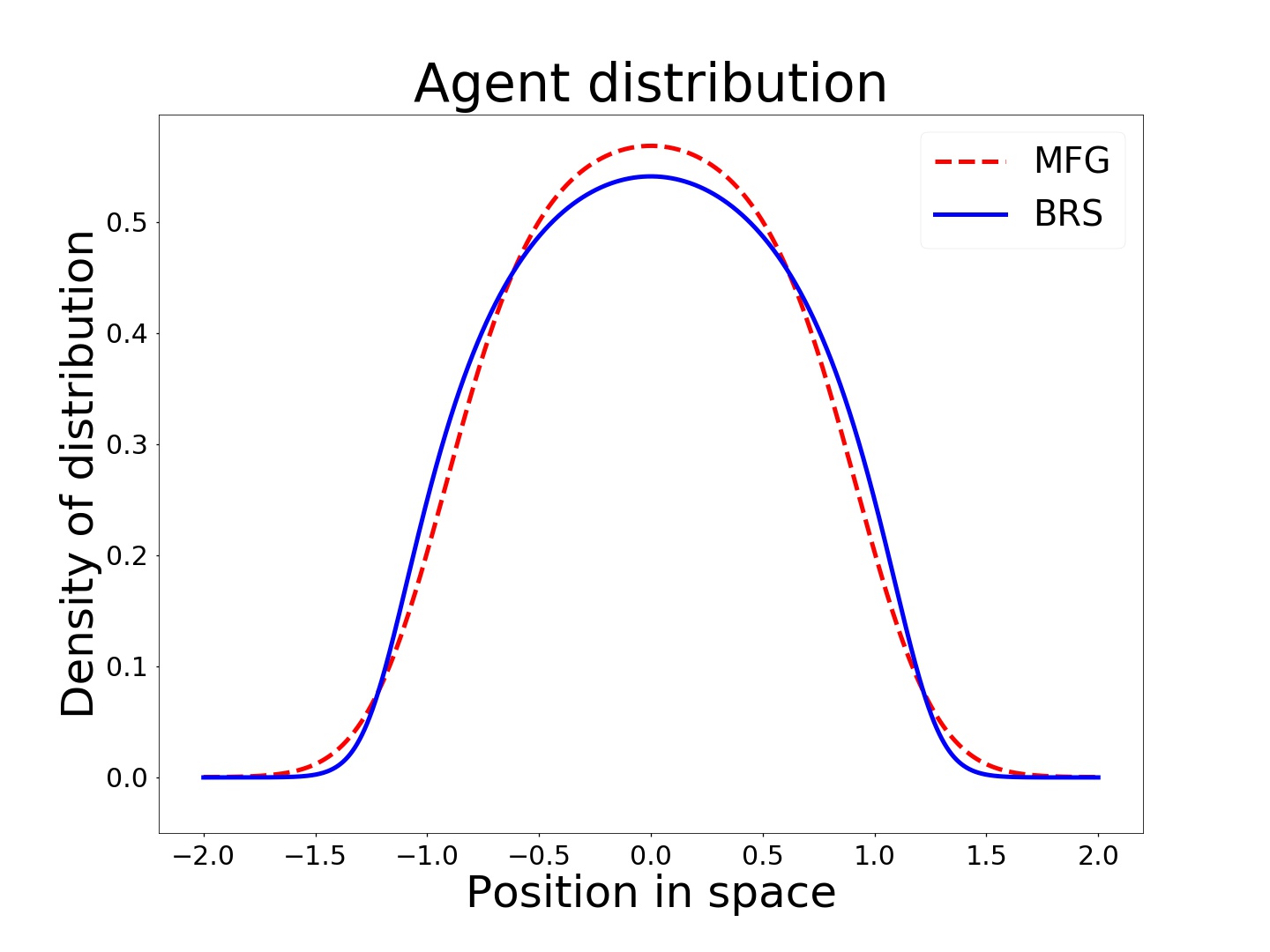}
	    \caption{$m_{\max} = 1$}
	    \label{fig:quad-max-1}
    \end{subfigure}
    \caption{Simulations of BRS and MFG with $h(x,m) = \frac{1}{m_{\max} - m} + x^2$}
    \label{fig:quad-max}
\end{figure}

\subsection{Single well potential}

In the first examples we investigate the behavior of solutions to both models for cost functionals of the form ${h(x,m) = F(m) + \beta x^2}$, using different functions $F$ and parameters $\beta$. We also analyse how the noise level $\sigma$ affects the two stationary states. Note that the case $F(m) = \log(m)$ was already discussed in Section~\ref{sec:quad potential}. We are particularly interested how penalising congestion by considering functions $F(m)$ of the form $F(m) = m^{\alpha}$ for some $\alpha > 0$ or $F(m) = \frac{1}{m_{\max} - m}$ for some $m_{\max} > \frac{1}{\Omega}$ affect solutions. The last choice introduces a `barrier' above which the density can not exceed. Using such a congestion term is more realistic from a modelling perspective than either the logarithmic or power--law term as it forces densities to stay below a certain physical reasonable limit. We will observe a similar dependence on the parameters $\beta,\sigma$ compared with the logarithmic congestion term --- and in fact the same can be said for all of our simulations. So for all values of $\sigma$ the MFG model responded less to changes in the strength of the potential $\beta$ compared with the BRS model, however the difference is most pronounced as $\sigma$ increases and again it may be expected that as $\sigma \to 0$ that the two models align very closely.

The most notable difference between the use of a logarithmic congestion term and a power--law congestion term is a difference in the shape of the distribution, particularly the flatness of the peak of the distribution as shown in figure~\ref{fig:quad-power}. Importantly this characteristic is shared by both the MFG and the BRS, suggesting the congestion terms $F(m)$ affect both models in similar ways. When looking at congestion terms of the form $F(m) = \frac{1}{m_{\max} - m}$, we can find regimes where the behaviour is similar to the logarithm, or more like a vastly exaggerated version of the power--law congestion. When $m_{\max}$ is large, as in figure~\ref{fig:quad-max-10}, the resulting distribution for both models looks like a normal distribution, similar to the case with logarithmic congestion. In fact when $m_{max}$ is very large, a formal asymptotic analysis using a Taylor expansion around $m_{\max}$ can be made which shows 
\[\frac{1}{m_{\max} - m} \approx \frac{1}{m_{\max}} \, .\]
Therefore the BRS  satisfies the equation
\[ m \approx \frac{1}{Z} e^{- \frac{2}{\sigma^2} \frac{ \beta m_{\max} x^2 - 1}{m_{\max}}} \approx \frac{1}{Z} e^{- \frac{2 \beta}{\sigma^2} x^2} \, , \]
with a normalisation constant $Z = \int_{\Omega} e^{- \frac{2\beta}{\sigma^2} x^2}$, which corresponds, on the whole space $\R$, to a normal distribution with zero mean and variance $\frac{\sigma^2}{4 \beta}$. The variance of the BRS solution found here differs from the logarithmic congestion case by $\frac{2}{4\beta}$. A similar analysis shows that when $m_{\max} \to \infty$ the MFG approximately resembles a normal distribution with zero mean and variance 
$\frac{\sigma^2}{2 \beta}$. In summary, solutions to the MFG and the BRS are both normal distributions with zero mean and with variances whose relative difference is $\frac{1}{2}$.

However, when $m_{\max}$ is reduced, as in figure~\ref{fig:quad-max-1}, the peak flattens out in a similar but exaggerated way compared to the power--law congestion. It is interesting to note that the BRS seems to respond more to the change in $m_{\max}$ than the MFG, this is in contrast to the role that $\sigma$ plays in the two models where the MFG responds more to changes in $\sigma$ compared with the BRS.

\begin{figure}[t]
    \begin{subfigure}{0.49 \textwidth}
        \centering
        \includegraphics[width=\linewidth]{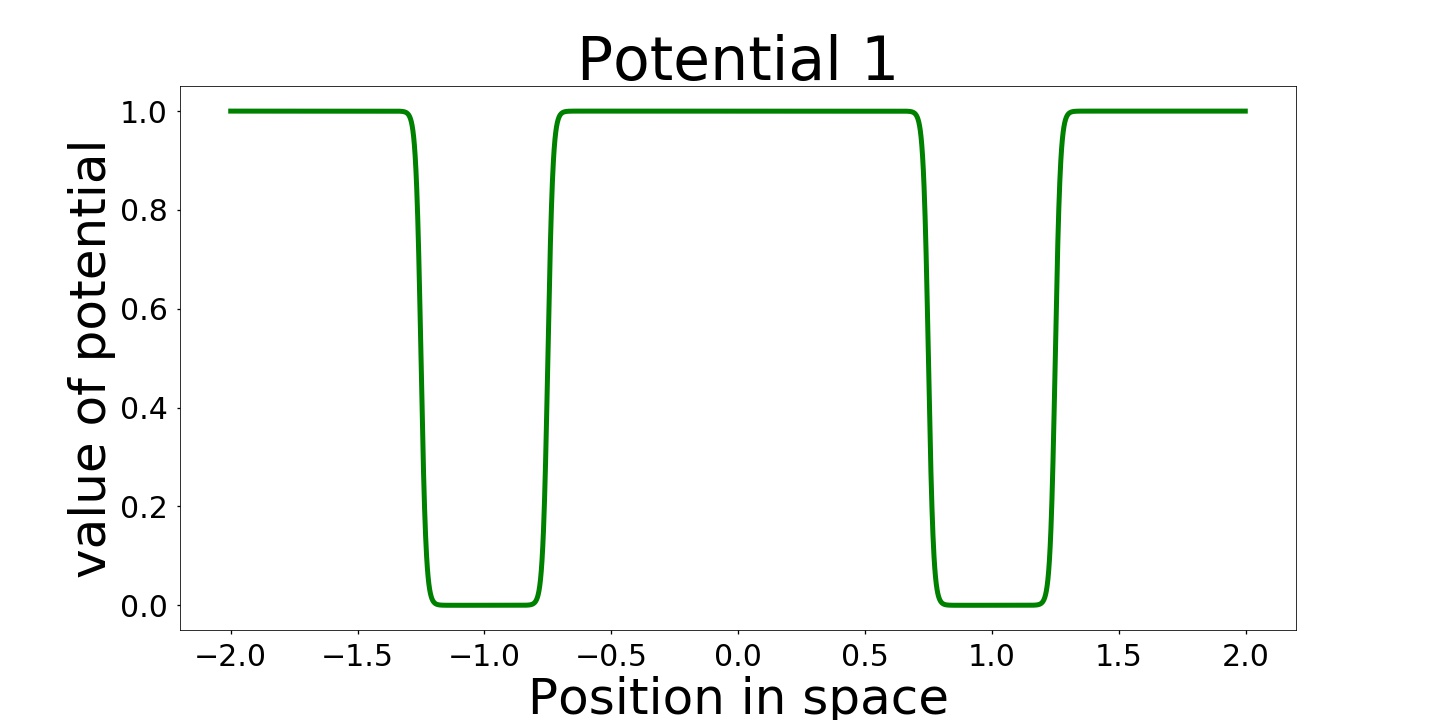}
        \caption{Double well potential for Figure~\ref{fig:double well 1}}
        \label{fig:double pot 1}
    \end{subfigure}
    \begin{subfigure}{0.49 \textwidth}
        \centering
        \includegraphics[width=\linewidth]{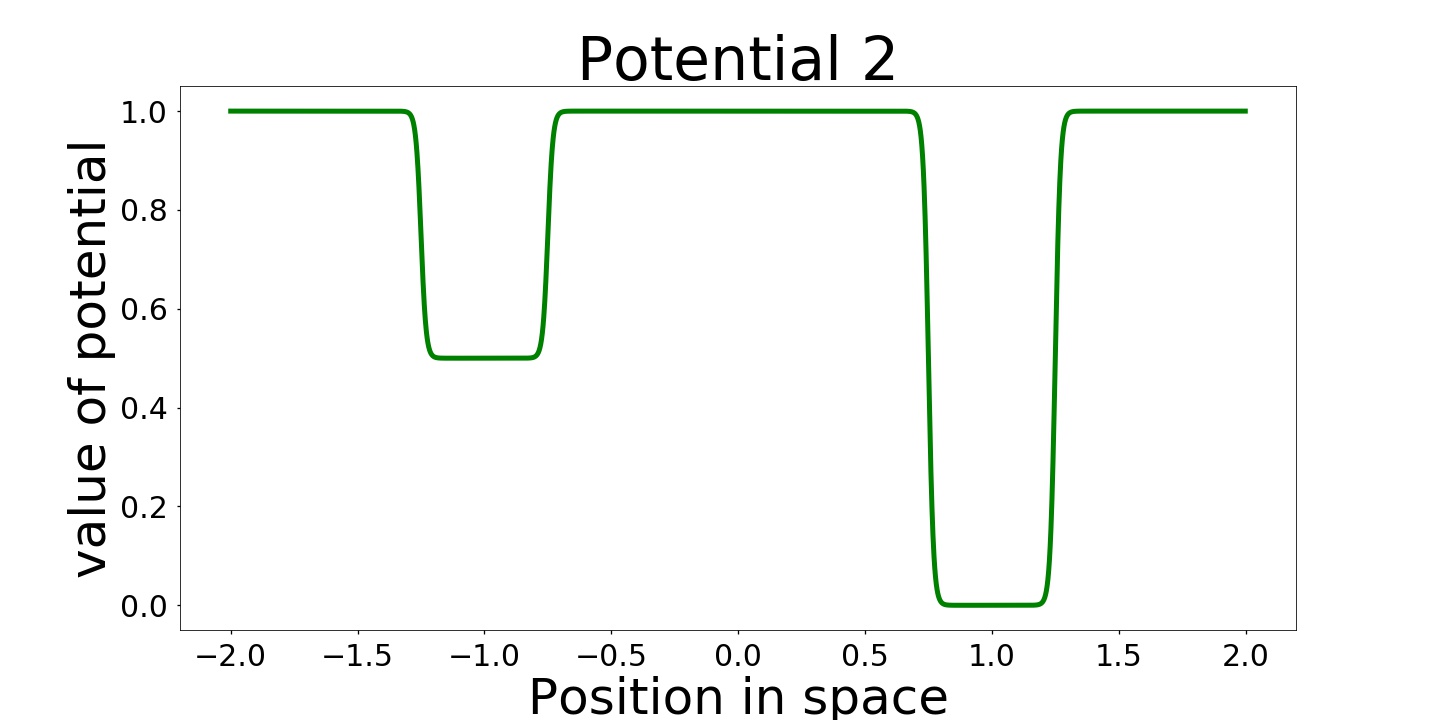}
        \caption{Double well potential for Figure~\ref{fig:double well 2}}
        \label{fig:double pot 2}
    \end{subfigure}
    
    \bigskip
    
    \begin{subfigure}{0.49 \textwidth}
        \centering
        \includegraphics[width=\linewidth]{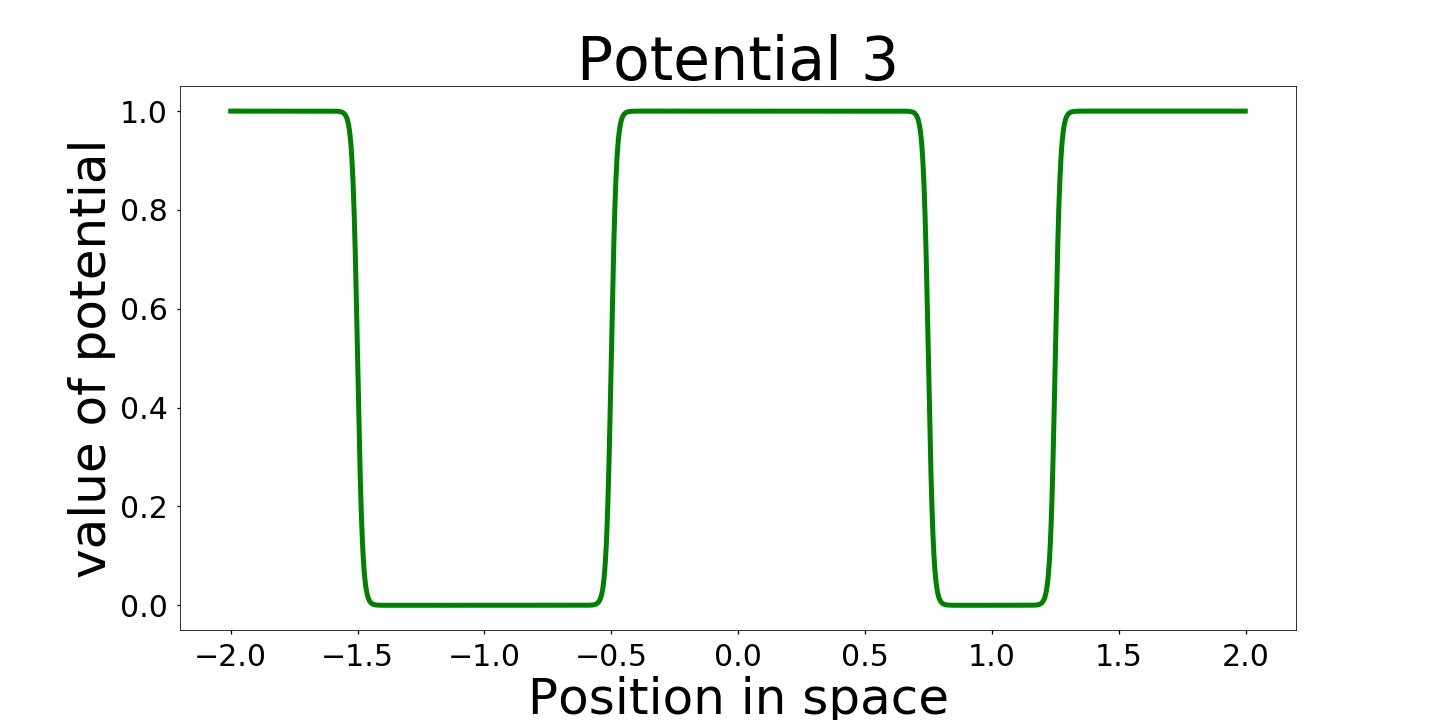}
        \caption{Double well potential for Figure~\ref{fig:double well 3}}
        \label{fig:double pot 3}
    \end{subfigure}
    \begin{subfigure}{0.49 \textwidth}
        \centering
        \includegraphics[width=\linewidth]{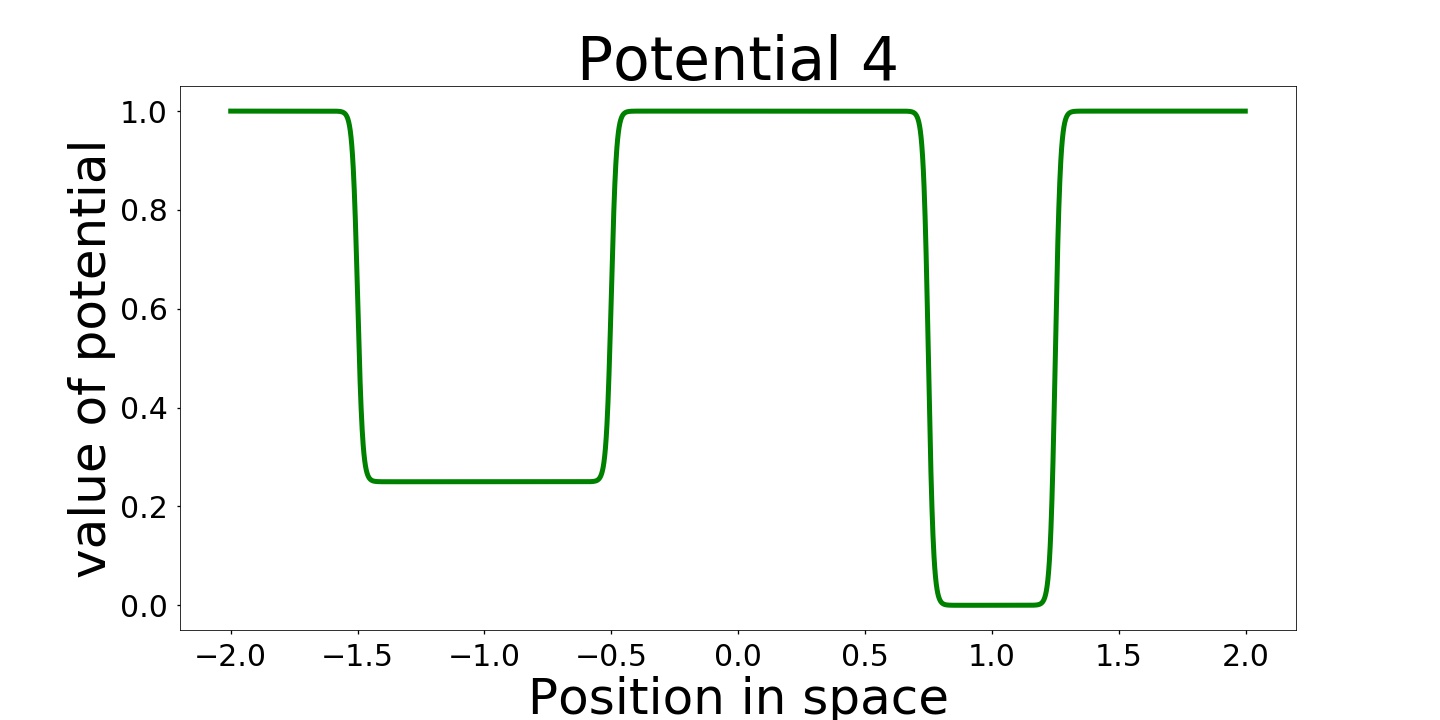}
        \caption{Double well potential for Figure~\ref{fig:double well 4}}
        \label{fig:double pot 4}
    \end{subfigure}
    
    \bigskip
    \centering
    \begin{subfigure}{0.49 \textwidth}
        \centering
        \includegraphics[width=\linewidth]{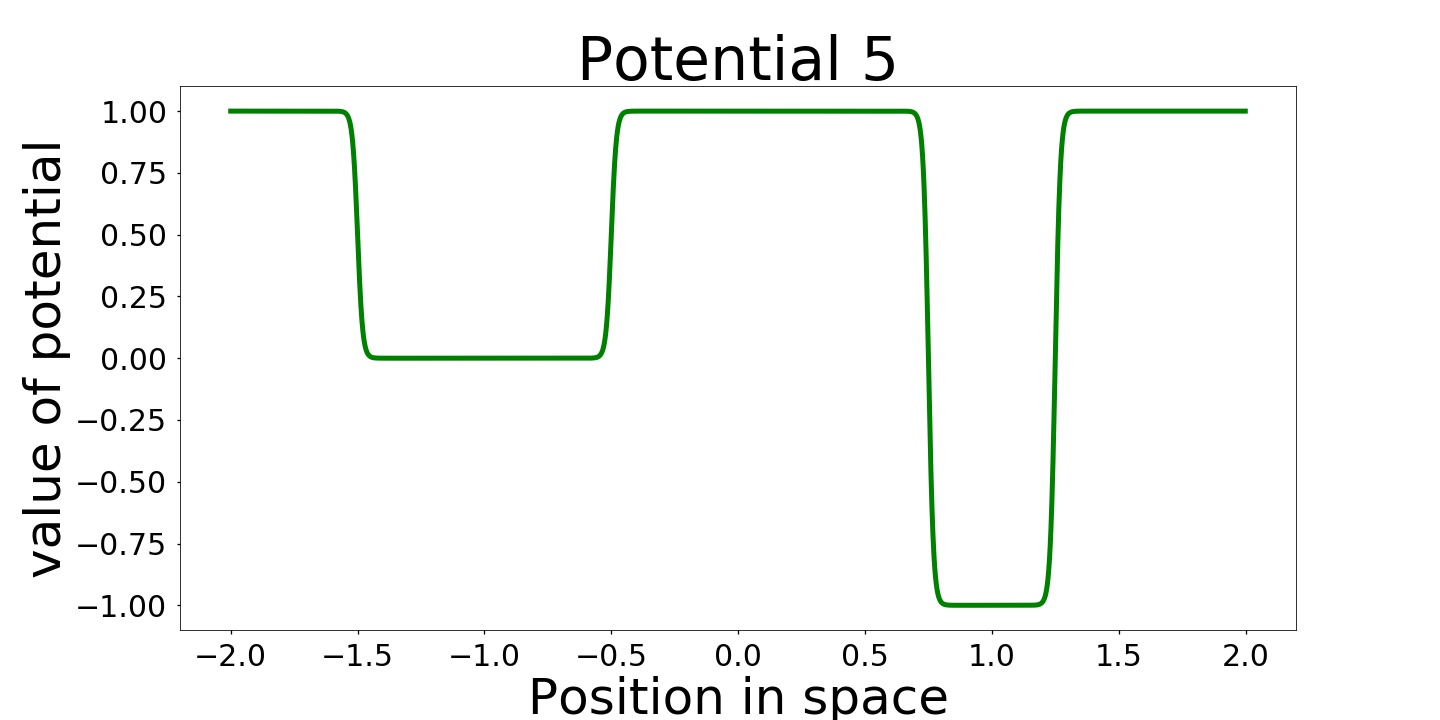}
        \caption{Double well potential for Figure~\ref{fig:double well 5}}
        \label{fig:double pot 5}
    \end{subfigure}
    \caption{Double well potentials for simulations}
    \label{fig:double pots}
\end{figure}

\subsection{Double well potential}

The previous subsection has given insight into how the form of the congestion term affects both models. In this section we explore how the potential term affects each model. For this section we will consider costs of the form $h(x,m) = F_1(x) + \log(m)$, where $F_1(x)$ will be a double well potential (see figure~\ref{fig:double pots}). Since the key insight of this section is to understand how varying $F_1$ affects the similarity of solutions to the BRS and MFG models, we have decided not to include results with different congestion terms other than the logarithm. From simulations it can be seen that the effect of changing the congestion term from $\log(m)$ to another term is very similar whether we are considering a single well or a double well.

Our simulations focus on five different double wells, which can be seen in figure~\ref{fig:double pots}. We vary the potentials as follows
\begin{enumerate}[topsep=0pt,itemsep=0pt,partopsep=2pt,parsep=1pt]
    \item same depth, same width,
    \item different depth, same width,
    \item same depth, different width,
    \item approximately similar perimeter,
    \item approximately similar volume.
\end{enumerate}

The simulations with the first two potentials, see figures~\ref{fig:double well 1} and~\ref{fig:double well 2}, where the widths of the two wells are always the same, show that the two models display similar qualitative behaviour. As with the single well potential, as $\sigma$ increases the discrepancy between the two models also increases, with the MFG model being more affected by the level of noise than the BRS. As expected, when the two wells are of equal depth (as in figure~\ref{fig:double well 1}) then both the MFG and BRS attribute equal weight between the wells, while when one well is deeper than the other (as in figure~\ref{fig:double well 2}) both models give more mass to the location of the deeper well. This is true for all values of $\sigma$, although the effect reduces as $\sigma$ increases, as can be seen by comparing figures~\ref{fig:double well 1a} and~\ref{fig:double well 2a} to figures~\ref{fig:double well 1b} and~\ref{fig:double well 2b} respectively.

Up to this point we have seen that the qualitative behaviour of the two models tends to agree. However when we look at double well potentials where the width of the well is varying, we start to observe differences. In figure~\ref{fig:double well 3}, where the wells have the same depth but different width, we see that the BRS still distributes density equally to the two wells. However the MFG model results in a higher density focussed in the wider well than in the narrower well. The reason the width has no effect on the BRS can be seen from studying the implicit equation. In each well we are solving $m = \frac{1}{Z} e^{- \frac{2}{\sigma^2} h(x,m)}$. In our case $h(x,m) = G(x) + F(m)$. Since the potential $G(x)$ is at the same depth in each well then the relative height of the distribution $m$ will be the same in each well. The reason the MFG is affected by the width of the of the well is that in finding the MFG solution we are in fact solving an elliptic equation to find the function $u$, hence at each point $x$ this $u$ will be affected by factors that can't be described by just looking at the value of $h$ at that point. In other words, the BRS depends only on local properties of the cost $h$ whereas the MFG depends also on non-local properties. To understand why the MFG assigns greater density to the wider well we need to look at the underlying optimisation problems related to the MFG and the BRS.

\begin{figure}[t]
	\begin{subfigure}{0.49 \textwidth}
        \centering
		\includegraphics[width=\textwidth]{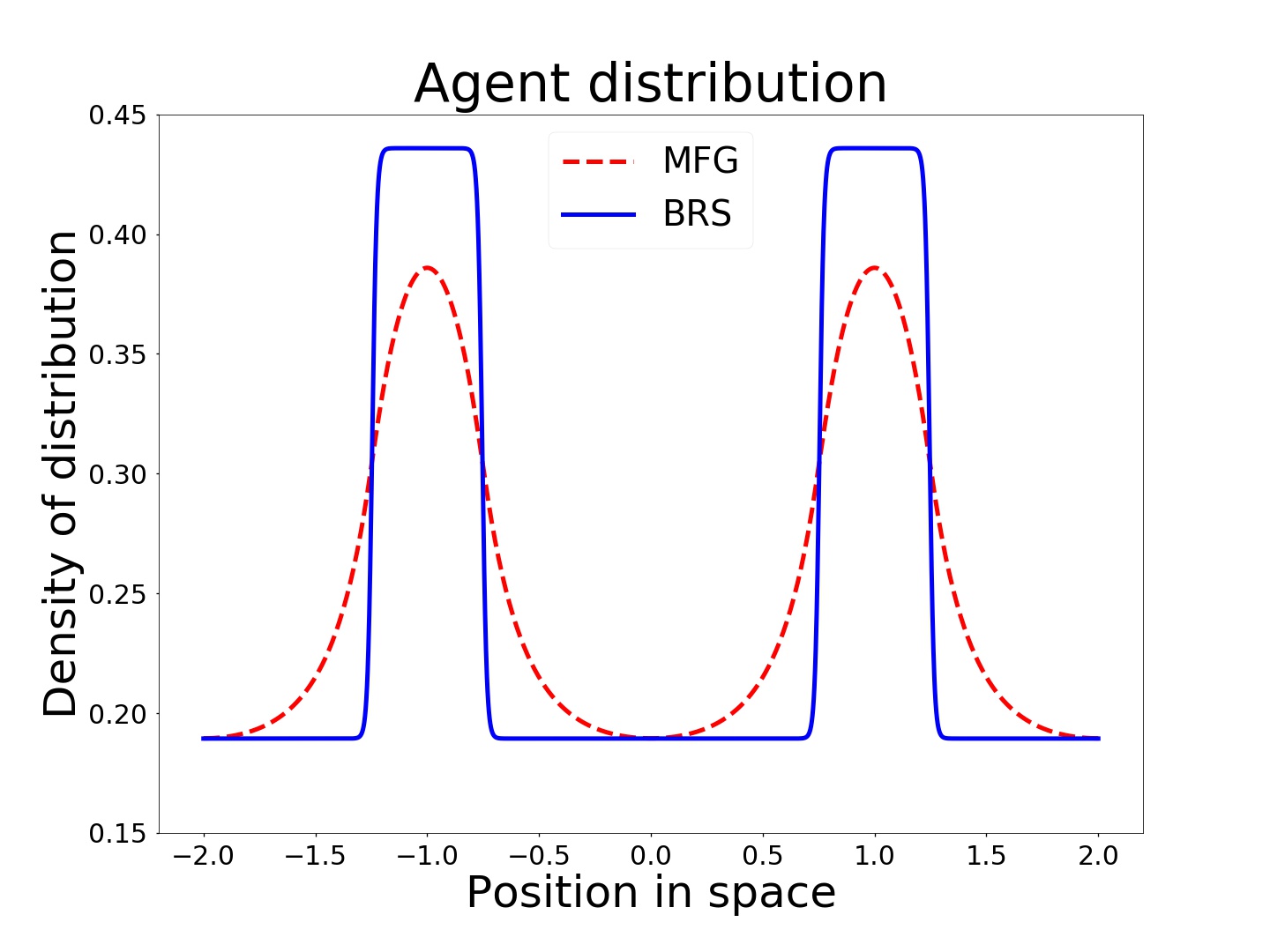}
    	\subcaption{$\frac{\sigma^2}{2} = 0.2$}
    	\label{fig:double well 1a}
	\end{subfigure}
	\begin{subfigure}{0.49 \textwidth}
        \centering
		\includegraphics[width=\textwidth]{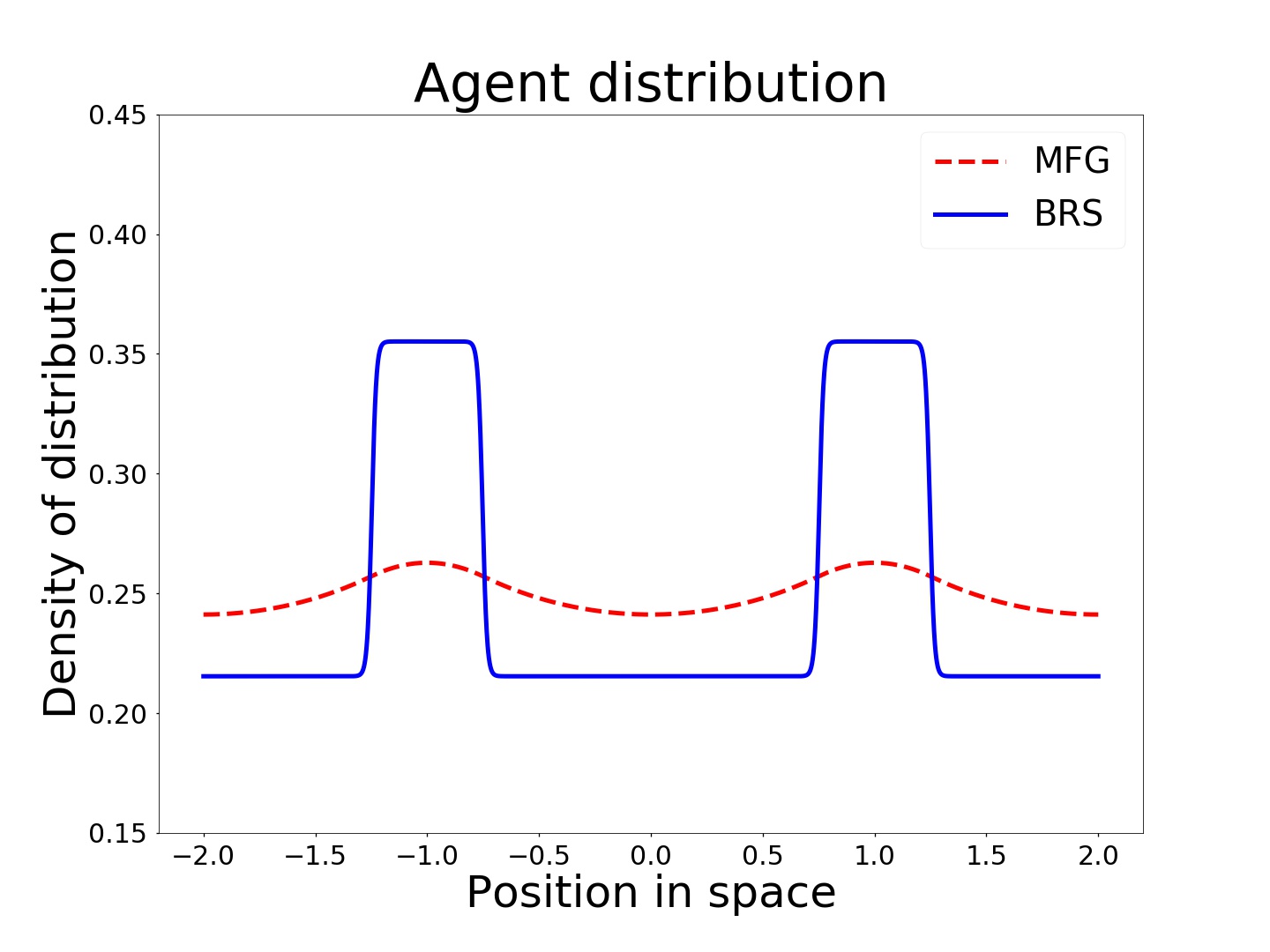}
    	\subcaption{$\frac{\sigma^2}{2} = 1$}
    	\label{fig:double well 1b}
	\end{subfigure}
	\caption{Simulation of BRS and MFG with logarithmic congestion and potential given in figure~\ref{fig:double pot 1}}
	\label{fig:double well 1}
\end{figure}
\begin{figure}[t]
	\begin{subfigure}{0.49 \textwidth}
        \centering
		\includegraphics[width=\textwidth]{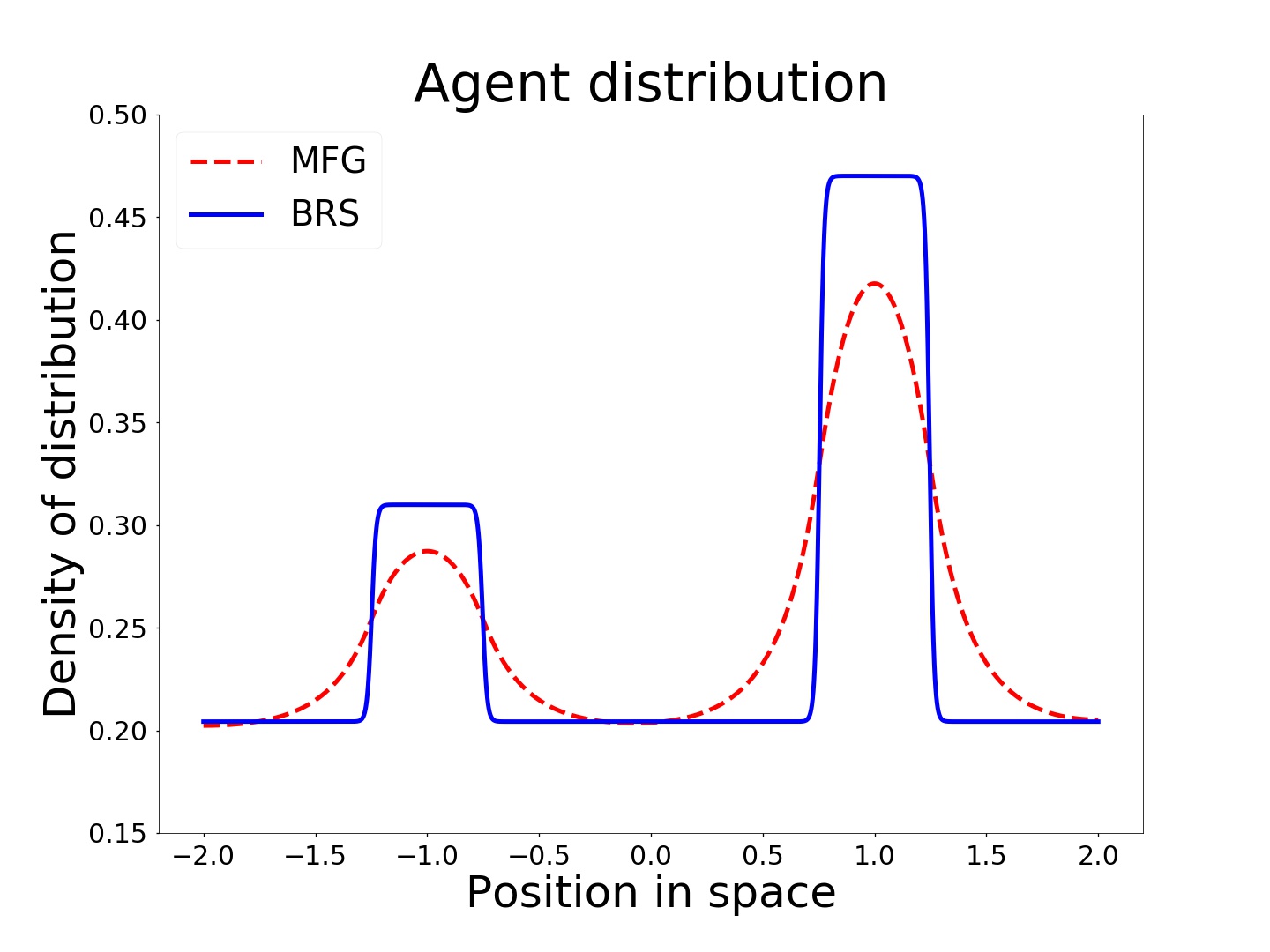}
    	\subcaption{$\frac{\sigma^2}{2} = 0.2$}
    	\label{fig:double well 2a}
	\end{subfigure}
	\begin{subfigure}{0.49 \textwidth}
        \centering
		\includegraphics[width=\textwidth]{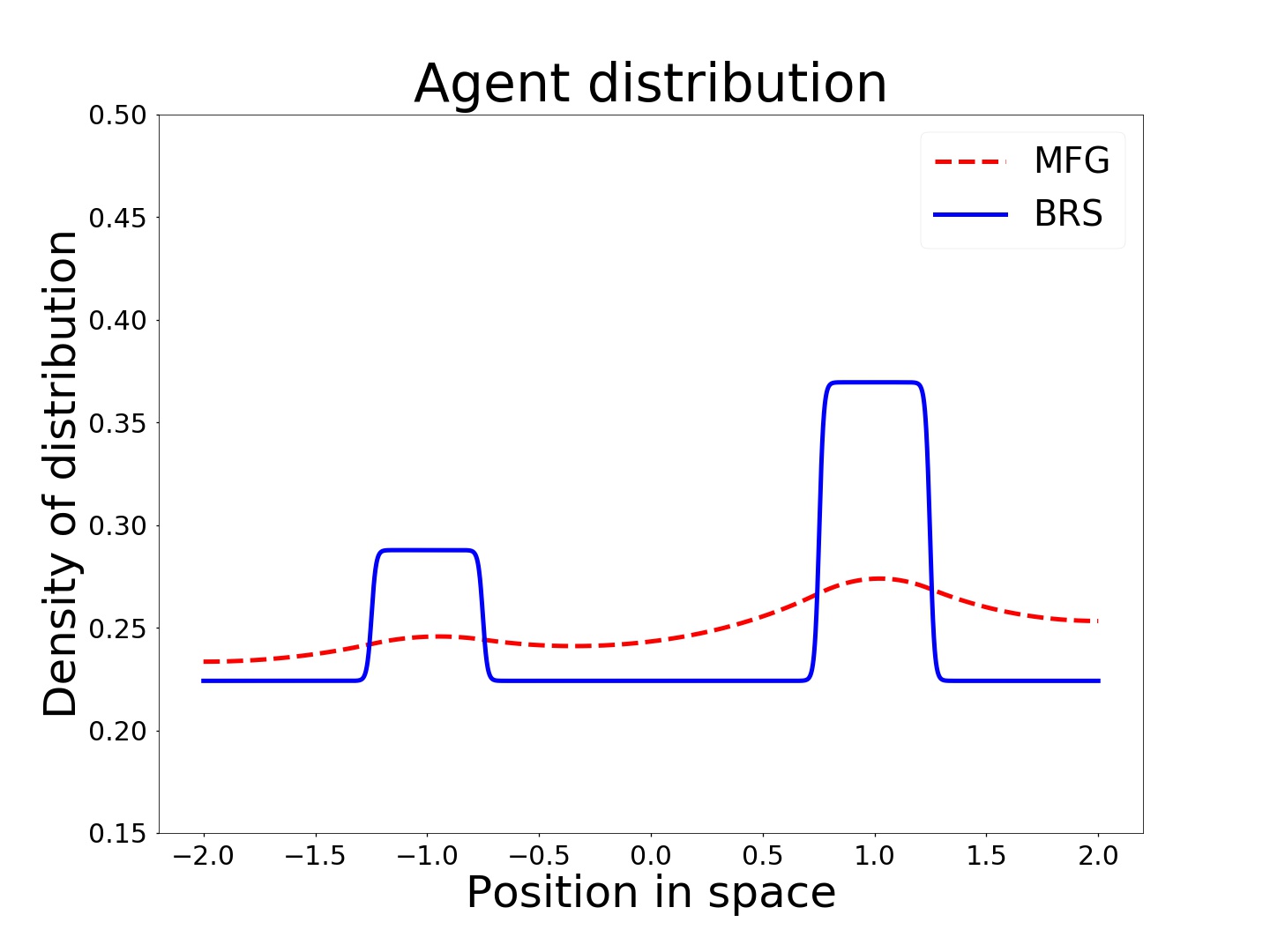}
    	\subcaption{$\frac{\sigma^2}{2} = 1$}
    	\label{fig:double well 2b}
	\end{subfigure}
	\caption{Simulation of BRS and MFG with logarithmic congestion and potential given in figure~\ref{fig:double pot 2}}
	\label{fig:double well 2}
\end{figure}

When considering the optimisation problems related to the dynamic MFG and BRS models and the long-term behaviour of these models, which results in the stationary models, we can see that the BRS model has no anticipation about the future system, while the MFG model does. Therefore agents in the MFG model are willing to incur higher congestion costs in the wider well as they can see that the cost to move out of the well to an area of lower congestion will be higher than the cost incurred for staying in the well (the cost functional being optimised has a quadratic running cost on the control). Since the cost for moving out of the well increases with the width of the well (as the wider the well either the longer an agent has to use their control, or the larger their control has to be), fewer agents are willing to move out of the wider well than the narrower in the long-run. Hence in the stationary case the wider well has a higher density associated to it than the narrower well. In contrast, we see that in going from the MFG to the BRS we renormalise the cost of the control by $\Delta t$ and take $\Delta t \to 0$ so the BRS doesn't consider the cost of moving along the width of the well in order to find an area of lower density. Therefore the BRS agents will not consider the width of the well when deciding whether to remain in it or leave. This further explains why the width of the well has no effect on the relative size of the density in each well for the BRS. 

We have seen that increasing well width affects only the MFG while increasing well depth affects both the MFG and BRS. Now we can balance these effects to create situations in which the two models give completely different results. Figure~\ref{fig:double well 4} involves a double well where the width and depth of the wells differ but the area of the well is the same, while in figure~\ref{fig:double well 5}  the perimeter of the wells was kept the same. In the case of a small noise term, then both the MFG and BRS favour the deeper wells. However with a larger noise term, figure~\ref{fig:double well 4b} shows that there are cases where the wider shallower well is favoured by the MFG while the narrower, deeper well is favoured by the BRS.

\begin{figure}[t]
	\begin{subfigure}{0.49 \textwidth}
        \centering
		\includegraphics[width=\textwidth]{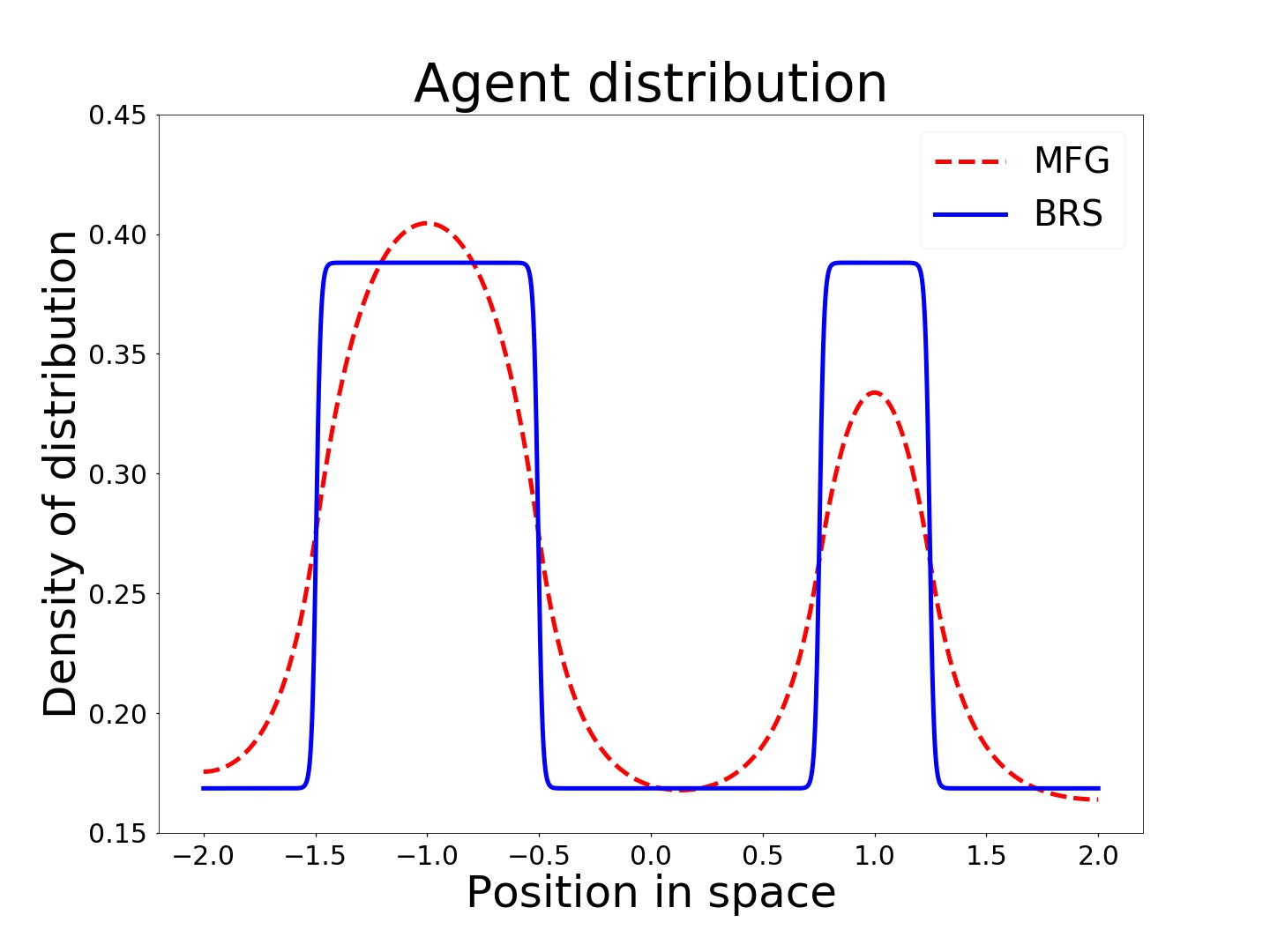}
    	\subcaption{$\frac{\sigma^2}{2} = 0.2$}
    	\label{fig:double well 3a}
	\end{subfigure}
	\begin{subfigure}{0.49 \textwidth}
        \centering
		\includegraphics[width=\textwidth]{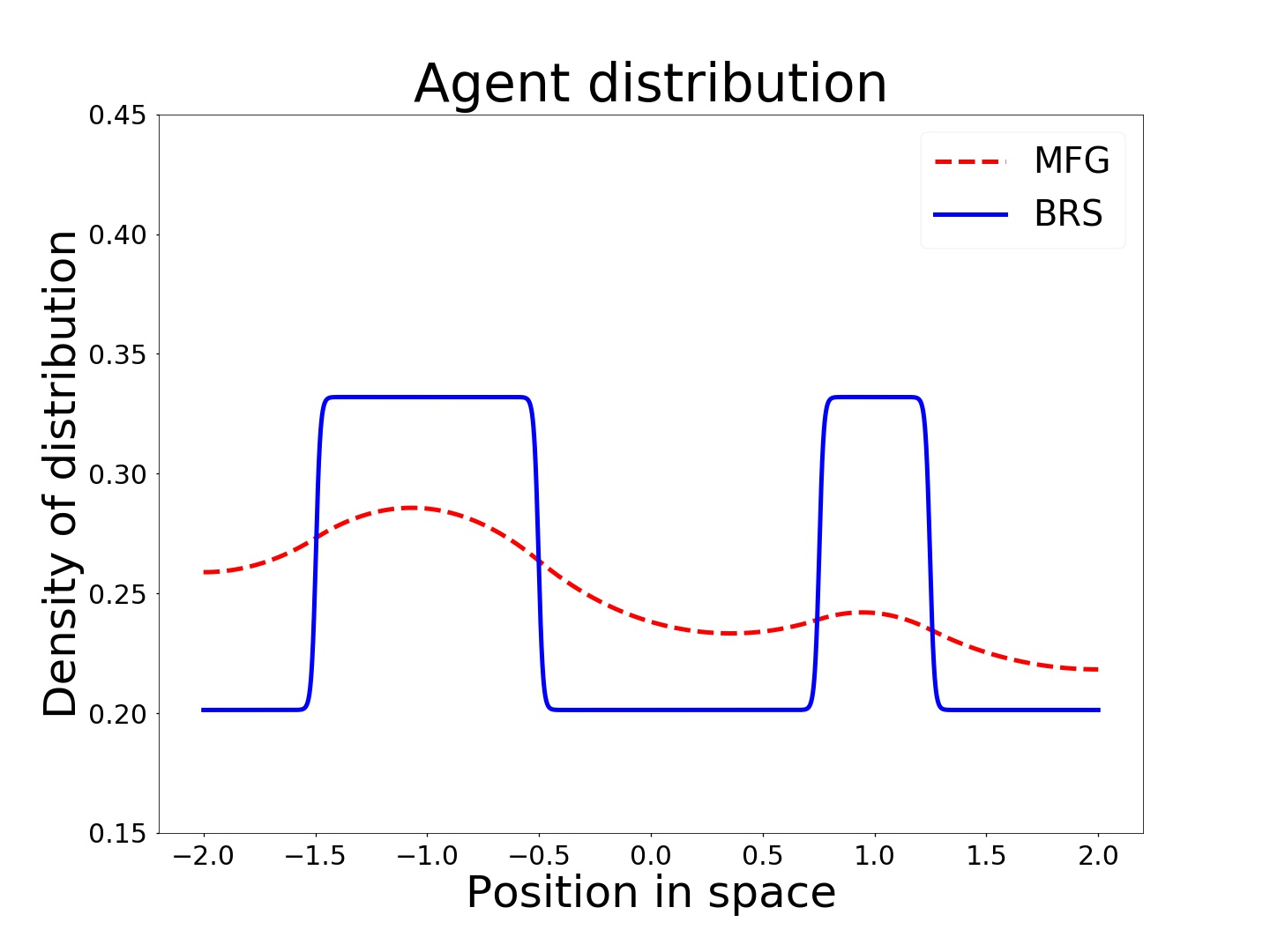}
    	\subcaption{$\frac{\sigma^2}{2} = 1$}
    	\label{fig:double well 3b}
	\end{subfigure}
	\caption{Simulation of BRS and MFG with logarithmic congestion and potential given in figure~\ref{fig:double pot 3}}
	\label{fig:double well 3}
\end{figure}
	
\begin{figure}[t]
	\begin{subfigure}{0.49 \textwidth}
        \centering
		\includegraphics[width=\textwidth]{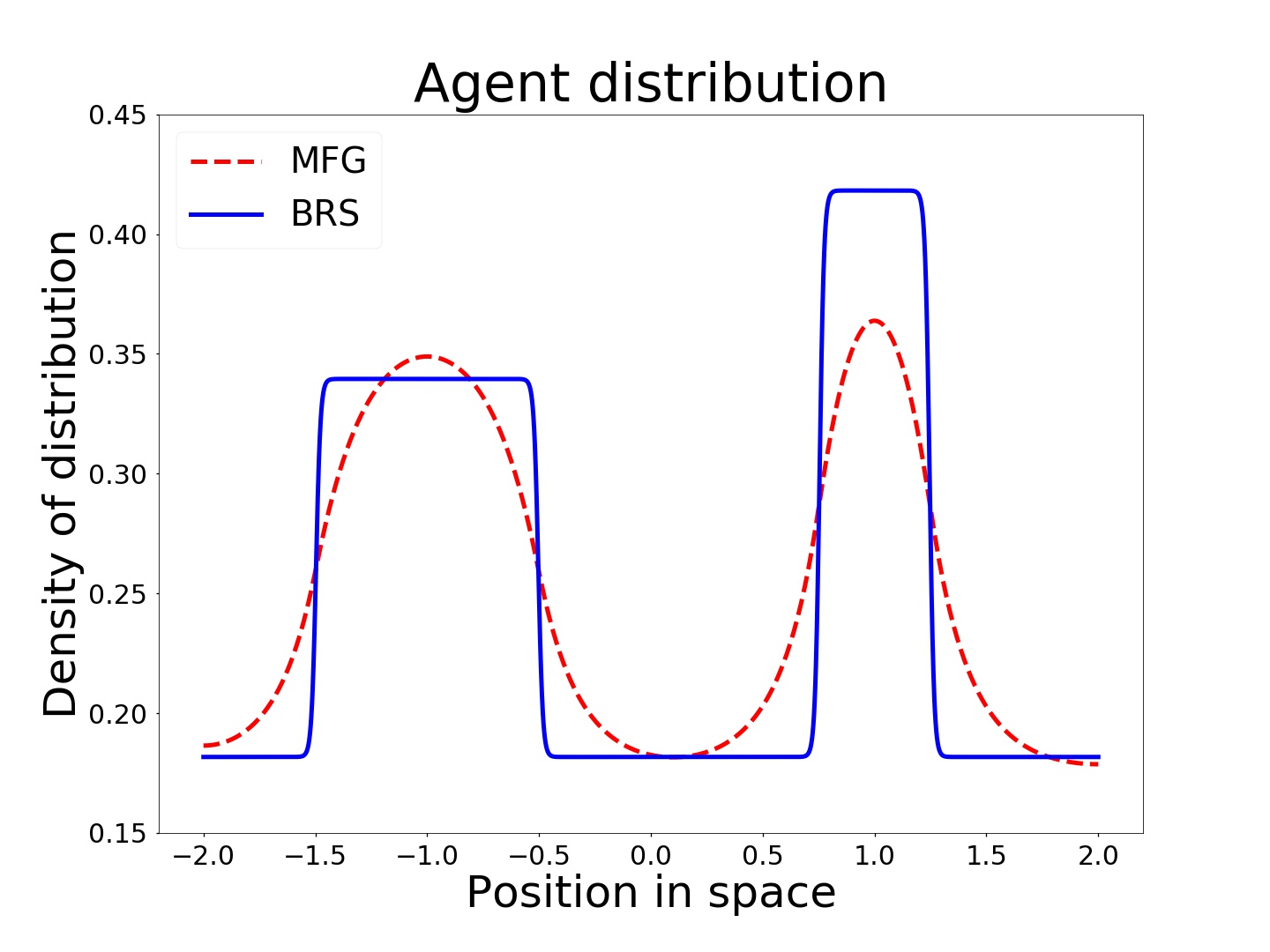}
    	\subcaption{$\frac{\sigma^2}{2} = 0.2$}
    	\label{fig:double well 4a}
	\end{subfigure}
	\begin{subfigure}{0.49 \textwidth}
        \centering
		\includegraphics[width=\textwidth]{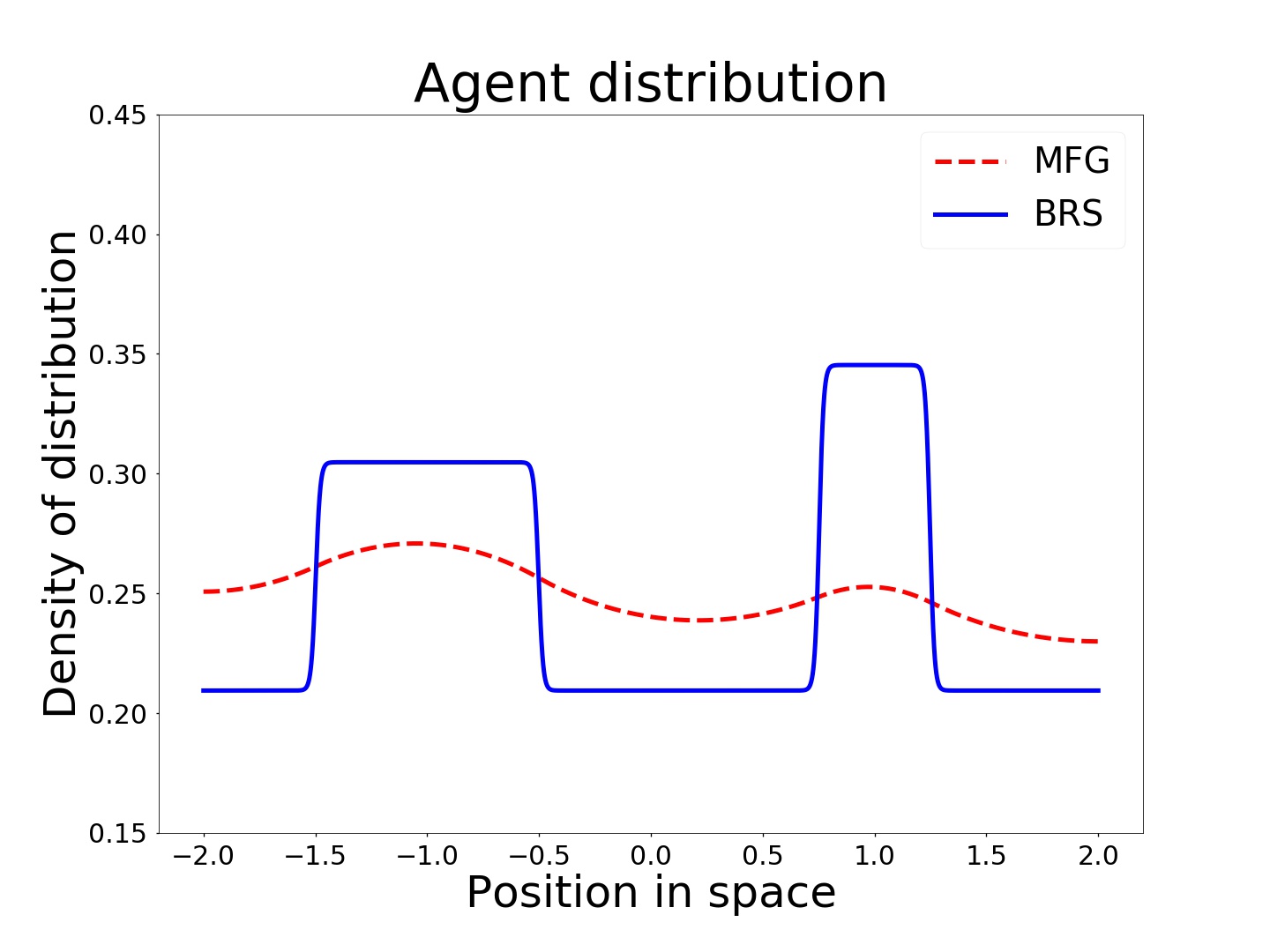}
    	\subcaption{$\frac{\sigma^2}{2} = 1$}
    	\label{fig:double well 4b}
	\end{subfigure}
	\caption{Simulation of BRS and MFG with logarithmic congestion and potential given in figure~\ref{fig:double pot 4}}
	\label{fig:double well 4}
\end{figure}

\section{Conclusion and outlook} \label{sec:conclusion}

\noindent In this paper we have systematically compared two models of interacting multi-agent systems in the stationary case. Through a proof of existence and uniqueness for each model we have seen that the BRS model can be reformulated as an implicit equation. This shows that the BRS model really only depends on local data of the cost function, while the MFG model, the solution of which is given by an elliptic equation, may have non-local dependenicies on the data. The existence and uniqueness proofs were based on the important assumption that the congestion term is increasing. However, the regularity requirements on the MFG data are less strict than those on the BRS data. Finally the proof gave an insight into the dependence of each model on the diffusion coefficient. We want to remark that the strategy of the proof is interesting on its own and that the only similar results presented in ~\cite{Cirant2015}, are based on different assumptions.\\
We supported our analytic results by numerical simulations and investigated the similarities and differences of the MFG and BRS models systematically in various computational experiments. We are planning to extend the analysis and simulations to the dynamic case in the future, and consider cost functions other than linear-quadratic ones.

\begin{figure}[t]
	\begin{subfigure}{0.49 \textwidth}
        \centering
		\includegraphics[width=\textwidth]{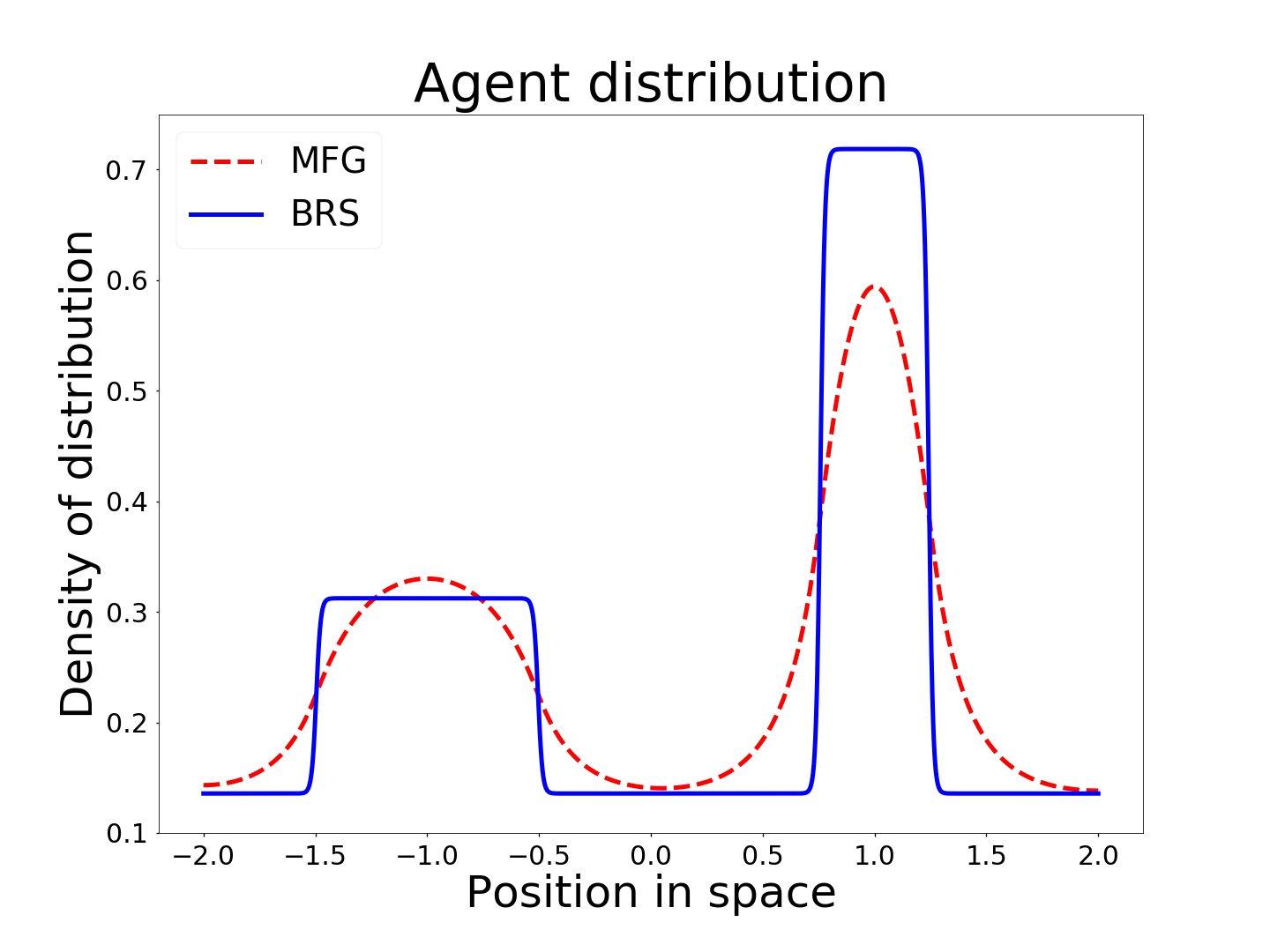}
    	\subcaption{$\frac{\sigma^2}{2} = 0.2$}
    	\label{fig:double well 5a}
	\end{subfigure}
	\begin{subfigure}{0.49 \textwidth}
        \centering
		\includegraphics[width=\textwidth]{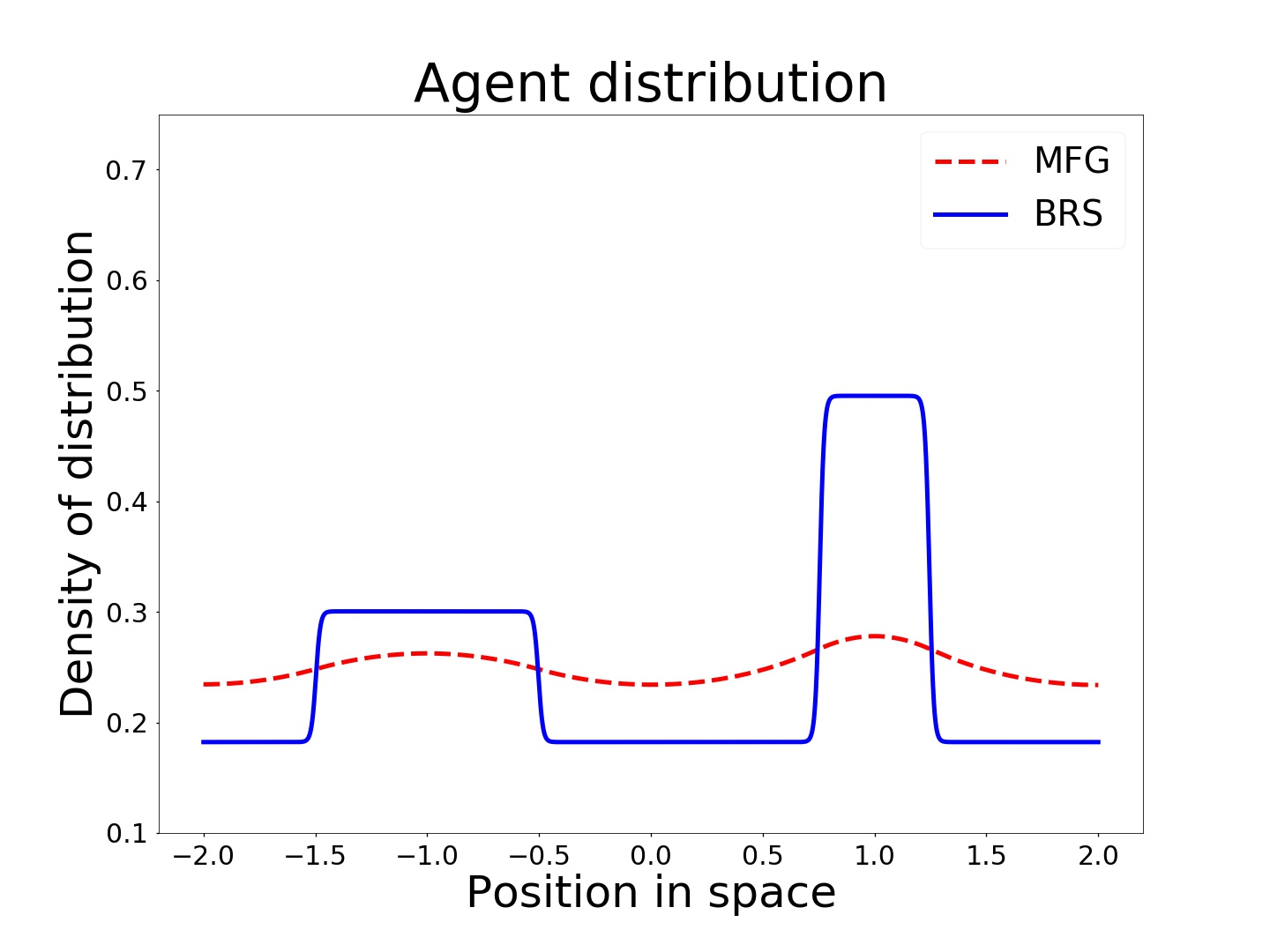}
    	\subcaption{$\frac{\sigma^2}{2} = 1$}
    	\label{fig:double well 5b}
	\end{subfigure}
	\caption{Simulation of BRS and MFG with logarithmic congestion and potential given in figure~\ref{fig:double pot 5}}
	\label{fig:double well 5}
\end{figure}

\bibliographystyle{abbrv}
\bibliography{BRS_MFG_comparison}

\end{document}